\documentclass[12pt]{article}%
\usepackage{amsmath,amsfonts, amssymb, graphicx}
\usepackage{amsmath}
\usepackage{amsfonts}
\usepackage{amssymb}
\usepackage{graphicx}%
\providecommand{\U}[1]{\protect\rule{.1in}{.1in}}
\newtheorem{theorem}{Theorem}[section]

\newtheorem{lemma}[theorem]{Lemma}

\newtheorem{corollary}[theorem]{Corollary}
\newcounter{remarknum}
\newenvironment{remark}{\addvspace{12pt}\refstepcounter{remarknum}
\noindent{\bf Remark \arabic{remarknum}.}}{\par\addvspace{12pt}}
\newenvironment{proof}{\addvspace{12pt}\noindent{\bf Proof:}}{
$\Box$\par\addvspace{12pt}}
\newcounter{examplenum}
\newenvironment{example}{\addvspace{12pt}\refstepcounter{examplenum}
\noindent{\bf Example \arabic{examplenum}.}}{\par\addvspace{12pt}}
\begin{document}

\title{ Non-cocompact Group Actions and $\pi_{1} $-Semistability at Infinity}
\author{Ross Geoghegan, Craig Guilbault\footnote{This research was supported in part by Simons Foundation Grants 207264 $\&$ 427244, CRG} 
 \hspace{.01in} and Michael Mihalik}
\date{\today}
\maketitle

\begin{abstract}
A finitely presented 1-ended group $G$ has {\it semistable fundamental group at infinity} if $G$ acts geometrically on a simply connected and locally compact ANR $Y$ having the property that any two proper rays in $Y$ are properly homotopic. This property of $Y$ captures a notion of connectivity at infinity stronger than ``1-ended", and is in fact a feature of $G$, being independent of choices. It is a fundamental property in the homotopical study of finitely presented groups. While many important classes of groups have been shown to have semistable fundamental group at infinity, the question of whether every $G$ has this property has been a recognized open question for nearly forty years. In this paper we attack the problem by considering a proper {\it but non-cocompact} action of a group $J$ on such an $Y$. This $J$ would typically be a subgroup of infinite index in the geometrically acting over-group $G$; for example $J$ might be infinite cyclic or some other subgroup whose semistability properties are known. We divide the semistability property of $G$ into a $J$-part and a ``perpendicular to $J$" part, and we analyze how these two parts fit together. Among other things, this analysis leads to a proof (in a companion paper \cite{M7}) that a class of groups previously considered to be likely counter examples do in fact have the semistability property.
\end{abstract}

\section{Introduction}

\label{Intro} In this paper we consider a new approach to the semistability
problem for finitely presented groups. This is a problem at the intersection
of group theory and topology. It has been solved for many classes of finitely
presented groups, for example \cite{BM91},\cite{Bo04}, \cite{GG12}, \cite{GM96},
\cite{LR}, \cite{M1}, \cite{M86}, \cite{M87}, \cite{MT1992}, \cite{MT92},
\cite{M6} - but not in general. We begin by stating

\medskip

\noindent\textbf{The Problem.} Consider a finitely presented infinite group
$G$ acting cocompactly by cell-permuting covering transformations on a 1-ended, simply
connected, locally finite $CW$ complex $Y$. Pick an expanding sequence
$\{C_{n}\}$ of compact subsets with $\operatorname*{int}C_{n}\subseteq
C_{n+1}$ and $\cup C_{n}=Y$, then choose a proper \textquotedblleft base
ray\textquotedblright\ $\omega:[0,\infty)\rightarrow Y$ with the property that
$\omega([n,n+1])$ lies in $Y-C_{n}$. Consider the inverse sequence
\begin{equation}
{\pi}_{1}(Y-C_{0},\omega (0))\overset{\lambda_{1}}{\longleftarrow}{\pi}_{1}%
(Y-C_{1}, \omega (1))\overset{\lambda_{2}}{\longleftarrow}{\pi}_{1}(Y-C_{2},\omega (3)%
)\overset{\lambda_{2}}{\longleftarrow}\cdots\label{representative of pro-pi1}%
\end{equation}
where the $\lambda_{i}$ are defined using subsegments of $\omega$. The Problem
is: EITHER to prove that this inverse sequence is always semistable, i.e. is
pro-isomorphic to a sequence with epimorphic bonding maps, OR to find a group
$G$ for which that statement is false. This problem is known to be independent
of the choice of $Y$, $\left\{  C_{n}\right\}  $, and $\omega$, and it is
equivalent to some more geometrical versions of semistability which we now recall.

A 1-ended, locally finite $CW$ complex $Y$, with proper base ray $\omega$,
\textit{has semistable fundamental group at $\infty$} if any of the following
equivalent conditions holds:

\begin{enumerate}
\item Sequence (\ref{representative of pro-pi1}) is pro-isomorphic to an
inverse sequence of surjections.

\item Given $n$ there exists $m$ such that, for any $q$, any loop in $Y-C_{m}
$ based at a point $\omega(t)$ can be homotoped in $Y-C_{n}$, with base point
traveling along $\omega$, to a loop in $Y-C_{q}$.

\item Any two proper rays in $Y$ are properly homotopic.
\end{enumerate}

Just as a basepoint is needed to define the fundamental group of a space, a base ray is needed to
define the fundamental pro-group at $\infty$. And just as a path between two basepoints defines
an isomorphism between the two fundamental groups, a proper homotopy between two base rays
defines a pro-isomorphism between the two fundamental pro-groups at $\infty$. In the absence of
such a proper homotopy it can happen that the two pro-groups are not pro-isomorphic (see
\cite{G}, Example 16.2.4.) Thus, in the case of $G$ acting cocompactly by covering
transformations as above, semistability is necessary and sufficient for the \textquotedblleft
fundamental pro-group at infinity of $G$\textquotedblright\ to be well-defined up to
pro-isomorphism.
\medskip

\noindent\textbf{The approach presented here.} In its simplest form our
approach is to restrict attention to the sub-action on $Y$ of an infinite finitely
generated subgroup $J$ having infinite index in $G$. We separate the
topology of $Y$ at infinity into 
\textquotedblleft the $J$-directions\textquotedblright\ and 
\textquotedblleft the directions in $Y$ orthogonal to $J$\textquotedblright, with the main result
being that, having
appropriate analogs of semistability in the two directions, implies that
$Y$ has semistable fundamental group at $\infty$.

For the purposes of an introduction, we first describe a special case of the
Main Theorem and give a few examples. A more far-reaching, but more technical,
version of the Main Theorem is given in Section \ref{coax}.
\medskip

Suppose $J$ is a finitely generated group acting by cell-permuting covering
transformations on a 1-ended locally finite and simply connected CW
complex $Y$. Let $\Gamma\left(  J,J^{0}\right)  $ be the Cayley graph of $J$ with
respect to a finite generating set $J^{0}$ and let $m:\Gamma\rightarrow Y$ be
a $J$-equivariant map. Then\medskip

\begin{itemize}
\item[a)] $J$ is \emph{semistable at infinity in }$Y$ if for any compact set
$C\subseteq Y$ there is a compact set $D\subseteq Y$ such that if $r$ and $s$
are two proper rays based at the same point in $\Gamma\left(  J,J^{0}\right)
-m^{-1}\left(  D\right)  $ then $mr$ and $ms$ are properly homotopic in $Y-C$
relative to $mr\left(  0\right)  =ms\left(  0\right)  $.

Standard methods show that the above property does not depend on the choice of finite
generating set $J^{0}$.

\item[b)] $J$ is \emph{co-semistable at infinity in }$Y$ if for any compact
set $C\subseteq Y$ there is a compact set $D\subseteq Y$ such that for any
proper ray $r$ in $Y-J\cdot D$ and any loop $\alpha$ based at $r(0)$ whose image lies in
$Y-D$, $\alpha$ can be pushed to infinity in $Y-C$ by a proper homotopy with the base point
tracking  $r$.
\medskip
\end{itemize}

\begin{theorem}
[Main Theorem---a special case]\label{Theorem: Special case of Main Theorem}If
$J$ is both semistable at infinity in $Y$ and co-semistable at infinity in
$Y$, then $Y$ has semistable fundamental group at infinity.\medskip
\end{theorem}

\begin{remark}
\label{Remark in Intro}

\begin{enumerate}
\item  To our knowledge, the theorems
proved here are the first non-obvious results that imply semistable
fundamental group at $\infty$ for a space $Y$ which might not admit a
cocompact action by covering transformations.

\item \label{Intro remark 2}In the special case where $J$ is an infinite
cyclic group, condition $(a)$ above is always satisfied since $\Gamma\left(
J,J^{0}\right)  $ can be chosen to be homeomorphic to $%
\mathbb{R};
$ any two proper rays in $%
\mathbb{R}
$ which begin at the same point and lie outside a nonempty compact subset of $%
\mathbb{R}
$ are properly homotopic in their own images. Moreover, since condition
$\left(  b\right)  $ is implied by the main hypothesis of \cite{GG12} (via
\cite[Lemma 3.1]{W92} or \cite[Th.16.3.4]{G}), Theorem
\ref{Theorem: Special case of Main Theorem} implies the main theorem of
\cite{GG12} .

\item The converse of Theorem \ref{Theorem: Special case of Main Theorem} is
trivial. If $Y$ is semistable at infinity and $J$ is any finitely generated
group acting as covering transformations on $Y$, it follows directly from the
definitions that $J$ is both semistable at infinity in $Y$ and co-semistable at
infinity in $Y$. So, our theorem effectively reduces checking the semistability
of the fundamental group at infinity of a space to separately checking two
strictly weaker conditions.

\item \label{Intro remark 4}In our more general version of Theorem
\ref{Theorem: Special case of Main Theorem} (not yet stated), the group $J$
will be permitted to vary for different choices of compact set $C$. No over-group containing these various groups is needed unless we want to extend our results to locally compact ANRs. That
issue is discussed in Corollary \ref{MC}.

\end{enumerate}
\end{remark}

\noindent\textbf{Some examples. }We now give four illuminating examples.
Admittedly, the conclusion of Theorem \ref{Theorem: Special case of Main Theorem} is known
by previous methods in the first three of these, but they are included because they nicely illustrate
how the semistability and co-semistability hypotheses lead to the semistability conclusion of the
Theorem. Moreover an understanding of these examples helps to motivate later proofs. In the
case of the fourth example the conclusion was not previously known.

\begin{example}
\label{Example: BS(1,2)}Let $G$ be the Baumslag-Solitar group $B\left(
1,2\right)  =\left\langle a,t\mid t^{-1}at=a^{2}\right\rangle $ acting by
covering transformations on $Y=T\times%
\mathbb{R}
$, where $T$ is the Bass-Serre tree corresponding to the standard graph of
groups representation of $G$, and let $J=\left\langle a\right\rangle \cong%
\mathbb{Z}
$. Then $J$ is semistable at infinity in $Y$ for the reasons described in
Remark \ref{Remark in Intro}(\ref{Intro remark 2}) above. To see that $J$ is
co-semistable at infinity in $Y$, choose $D\subseteq Y$ to be of the form
$T_{0}\times\left[  -n,n\right]  $, where $n\geq1$ and $T_{0}$ is a finite
subtree containing the \textquotedblleft origin\textquotedblright\ $0$ of $T$.
Then each component of $Y-J\cdot D$ is simply connected (it is a subtree
crossed with $%
\mathbb{R}
$). So pushing $\alpha$ to infinity along $r$ can be accomplished by first
contracting $\alpha$ to its basepoint, then sliding that basepoint along $r$
to infinity.
\end{example}

\vspace{.4in} \vbox to 2in{\vspace {-.3in} \hspace {.2in}
\hspace{.1 in}
\includegraphics[scale=.4]{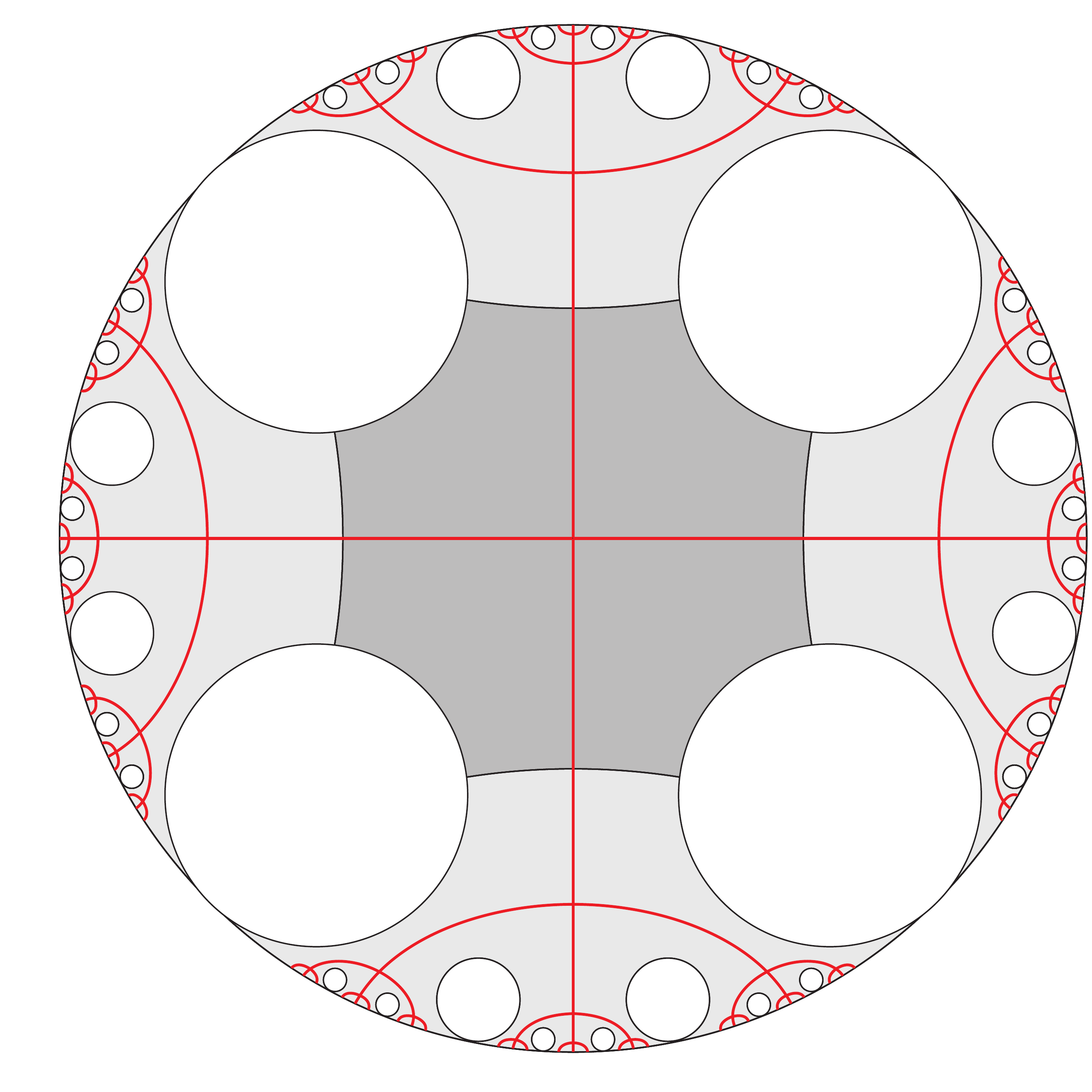}
\vss }

\vspace{2.3in}

\centerline{Figure 1}

\medskip

\begin{example}
\label{Example: F_2}Let $J=\left\langle a,b\mid\right\rangle $ be the
fundamental group of a punctured torus of constant curvature $-1$ and consider
the corresponding action of $J$ on $Y=\mathbb{H}^{2}$. Figure 1 shows
$\mathbb{H}^{2}$ with an embedded tree representing the image of a well-chosen
$m:\Gamma\left(  J,\left\{  a,b\right\}  \right)  \rightarrow\mathbb{H}^{2}$.
The shaded region represents a typical $J\cdot D$ for a carefully chosen
compact $D\subseteq\mathbb{H}^{2}$, which is represented by the darker
shading. The components of $\mathbb{H}^{2}-J\cdot D$ are open horoballs.
Notice that two proper rays in $\Gamma\left(  J,\left\{  a,b\right\}  \right)
-m^{-1}\left(  D\right)  $, which begin at the same point, are not necessarily
properly homotopic in $\Gamma\left(  J,\left\{  a,b\right\}  \right)
-m^{-1}\left(  D\right)  $, but their images are properly homotopic in
$\mathbb{H}^{2}-D$; so $J$ is semistable at infinity in $\mathbb{H}^{2}$.
Moreover, since each component of $\mathbb{H}^{2}-J\cdot D$ is simply
connected, $J$ is co-semistable at infinity in $\mathbb{H}^{2}$ for the same
reason as in Example \ref{Example: BS(1,2)}.
\end{example}

\begin{example}
\label{Example: Figure-eight}Let $K\subseteq S^{3}$ be a figure-eight knot;
endow $S^{3}-K$ with a hyperbolic metric; and consider the corresponding
proper action of the knot group $J$ on $\widetilde{S^{3}-K}=\mathbb{H}^{3}$.
Much like the previous example, there exists a nice geometric embedding of a
Cayley graph of $J$ into $\mathbb{H}^{3}$ and choices of compact
$D\subseteq\mathbb{H}^{3}$ so that $\mathbb{H}^{3}-J\cdot D$ is an infinite
collection of (3-dimensional) open horoballs. Since $J$ itself is known to be
$1$-ended with semistable fundamental group at infinity (a useful case to keep
in mind), the first condition of Theorem
\ref{Theorem: Special case of Main Theorem} is immediate. And again,
co-semistability at infinity follows from the simple connectivity of the
horoballs.\medskip
\end{example}

\begin{example}
\label{Ascending HNN extension}For many years an outstanding class of finitely presented
groups not known to be semistable at $\infty$ has been
the class of finitely presented ascending HNN extensions whose base groups
are finitely generated but not finitely presented\footnote{The case of
finitely presented base group was settled long ago in \cite{HNN1}.}. While
Theorem \ref{MT} does not establish semistability for this whole class, it
does so for a significant subclass --- those of \textquotedblleft finite
depth\textquotedblright. This new result is established in \cite{M7}, a paper which makes use of
the more technical Main Theorem \ref{MT} proved here. In particular, allowing the group $J$ to
vary (see Remark \ref{Remark in Intro}(\ref{Intro remark 4})) is important in this example.
\medskip
\end{example}

\noindent\textbf{Outline of the paper.} The paper is
organized as follows. We consider 1-ended simply connected locally finite
CW complexes $Y$, and groups $J$ that act on $Y$ as covering transformations.
In $\S $\ref{sss} we review a number of equivalent definitions for a space and
group to have semistable fundamental group at $\infty$. In $\S $\ref{coax} we state our Main Theorem \ref{MT} in full generality and 
formally introduce the two somewhat orthogonal notions in the hypotheses of
Theorem \ref{MT}. The first is that of a finitely generated group $J$ being
semistable at $\infty$ in $Y$ with respect to a compact set $C$, and the
second defines what it means for $J$ to be co-semistable at $\infty$ in $Y$
with respect to $C$. In $\S $\ref{out} we give a geometrical outline and
overview of the proof of the main theorem. In $\S $\ref{stabil} we prove a
number of foundational results. Suppose $C$ is a compact subset of $Y$ and $J$
is a finitely generated group acting as covering transformations on $Y$.
Define $J\cdot C$ to be $\cup_{j\in J}j(C)$. We consider components $U$ of
$Y-J\cdot C$ such that the image of $U$ in $J\backslash Y$ is not contained in
a compact set. We call such $U$, $J$-unbounded. We show there are only
finitely many $J$-unbounded components of $Y-J\cdot C$, up to translation in
$J$ and the $J$-stabilizer of a $J$-unbounded component is an infinite group.
In $\S $\ref{SVK} we use van Kampen's Theorem to show that for a finite
subcomplex $C$ of $Y$, the $J$-stabilizer of a $J$-unbounded component of
$Y-J\cdot C$ is a finitely generated group. A bijection between the ends of
the stabilizer of a $J$-unbounded component of $Y-J\cdot C$ and
\textquotedblleft$J$-bounded ends" of that component is produced in
$\S $\ref{bij}. The constants that arise in our bijection are shown to be
$J$-equivariant. In $\S $\ref{proof} we prove our main theorem. A generalization of our main theorem from CW complexes to absolute neighborhood retracts is proved in $\S $\ref{ANR}.

\section{Equivalent definitions of semistability}\label{sss}

Some equivalent forms of
semistability have been stated in the Introduction. It will be convenient to have the following:

\begin{theorem}
\label{ssequiv} (see Theorem 3.2\cite{CM2}) With $Y$ as before, the following are equivalent:
\begin{enumerate}
\item $Y$ has semistable fundamental group at $\infty$.
\item Let $r:[0,\infty)\to Y$ be a proper base ray. Then for any compact
set $C$ there is a compact set $D$ such that for any third compact set $E$
and loop $\alpha$ based at $r(0)$ whose image lies in $Y-D$, $\alpha$ is homotopic
to a loop in $Y-E$, by a homotopy with image in $Y-C$, where $\alpha $ tracks $r$.
\item For any compact set $C$ there is a compact set $D$ such that if $r$ and
$s$ are proper rays based at $v$ and with image in $Y-D$, then $r$ and $s $
are properly homotopic $rel\{v\}$ by a proper homotopy supported in $Y-C$.
\item If $C$ is compact in $Y$ there is a compact set $D$ in $Y$ such that for
any third compact set $E$ and proper rays $r$ and $s$ based at a vertex $v$ and with
image in $Y-D$, there is a path $\alpha$ in $Y-E$ connecting points of $r$ and
$s$ such that the loop determined by $\alpha$ and the initial segments of $r$
and $s$ is homotopically trivial in $Y-C$.
\end{enumerate}
\end{theorem}

\begin{proof} That the first three conditions are equivalent is shown in Theorem 3.2 of
\cite{CM2}. Condition $4$ is clearly equivalent to  the more standard  Condition $3$.
\end{proof}

\section{The Main Theorm and its definitions}
\label{coax} We are now ready to state our main theorem in its general form.
After doing so, we will provide a detailed discussion of the definitions that
go into that theorem. Both the theorem and the definitions generalize those
found in the introduction.

\begin{theorem}
[Main Theorem]\label{MT} Let $Y$ be a 1-ended simply connected locally finite
CW complex. Assume that for each compact subset $C_{0}$ of $Y$ there is a
finitely generated group $J$ acting as cell preservering covering
transformations on $Y$, so that (a) $J$ is semistable at $\infty$ in $Y$ with
respect to $C_{0}$, and (b) $J$ is co-semistable at $\infty$ in $Y$ with
respect to $C_{0}$. Then $Y$ has semistable fundamental group at $\infty$.
\end{theorem}

\begin{remark}
If there is a group $G$ (not necessarily finitely generated) acting as
covering transformations on $Y$ such that each of the groups $J$ of Theorem
\ref{MT} is isomorphic to a subgroup of $G$, then the condition that $Y$ is a
locally finite CW complex can be relaxed to: $Y$ is a locally compact absolute
neighborhood retract (ANR) (see Corollary \ref{MC}).
\end{remark}

The {\it distance} between vertices of a CW complex will always
be the number of edges in a shortest edge path connecting them. The space $Y$
is a 1-ended simply connected locally finite CW complex, and for each compact
subset $C_{0}$ of $Y$, $J(C_{0})$ is an infinite finitely generated group
acting as covering transformations on $Y$ and preserving some locally finite
cell structure on $Y$. Fix $\ast$ a base vertex in $Y$. Let $J^{0}$ be a
finite generating set for $J$ and $\Lambda(J,J^{0})$ be the Cayley graph of
$J$ with respect to $J^{0}$. Let $z_{(J,J^{0})}:(\Lambda(J,J^{0}%
),1)\rightarrow(Y,\ast)$ be a $J$-equivariant map so that each edge of
$\Lambda$ is mapped to an edge path of length $\leq K(J^{0})$. If $r$ is an
edge path in $\Lambda$, then $z(r)$ is called a $\Lambda$\textit{-path} in
$Y$. The vertices $J\ast$ are called $J$\textit{-vertices}.

If $C_{0}$ is a compact subset of $Y$ then the group $J$ is \textit{semistable
at $\infty$ in $Y$ with respect to $C_{0}$} if there exists a compact set $C$
in $Y$ and some (equivalently any) finite generating set $J^{0}$ for $J$ such
that for any third compact set $D$ and proper edge path rays $r$ and $s$ in
$\Lambda(J,J^{0})$ which are based at the same vertex $v$ and are such that
$z(r)$ and $z(s)$ have image in $Y-C$ then there is a path $\delta$ in $Y-D$
connecting $z(r)$ and $z(s)$ such that the loop determined by $\delta$ and the
initial segments of $z(r)$ and $z(s)$ is homotopically trivial in $Y-C_{0}$
(compare to Theorem \ref{ssequiv}(4)).

Note that this definition requires less than one requiring $z(r)$ and $z(s)$
be properly homotopic $rel\{z(v)\}$ in $Y-C_{0}$ (compare to Theorem
\ref{ssequiv}(3)). It may be that the path $\delta$ is not homotopic to a path
in the image of $z$ by a homotopy in $Y-C_{0}$. This definition is independent
of generating set $J^{0}$ and base point $\ast$ by a standard argument,
although $C$ may change as $J^{0}$, $\ast$ and $z$ do. When $J$ is semistable
at infinity in $Y$ with respect to $C_{0}$, we may say $J$ is
\textit{semistable at $\infty$ in $Y$ with respect to $J^{0}$, $C_{0}$, $C$
and $z$}. Observe that if $\hat C$ is compact containing $C$ then $J$ is also
semistable at $\infty$ in $Y$ with respect to $J^{0}$, $C_{0}$, $\hat C$ and
$z$.

If $J$ is 1-ended and semistable at $\infty$ or 2-ended, then $J$ is always
semistable at $\infty$ in $Y$ with respect to any compact subset $C_{0}$ of
$Y$. The semistability of the fundamental group at $\infty$ of a locally
finite CW complex only depends on the 2-skeleton of the complex (see for
example, Lemma 3 \cite{LR}). Similarly, the semistability at $\infty$ of a
group in a CW complex only depends on the 2-skeleton of the complex.

The notion of $J$ being co-semistable at infinity in a space $Y$ is a bit
technical, but has its roots in a simple idea that is fundamental to the main
theorems of \cite{GG12} and \cite{W92}. in both of these papers 
 $J$ is an infinite cyclic group acting as covering transformations on a 1-ended
simply connected space $Y$ with pro-monomorphic fundamental group at $\infty$.
Wright \cite{W92} showed that under these conditions the following could be proved:

$(\ast)$ Given any compact set $C_{0}\subset Y$ there is a compact set
$C\subset Y$ such that any loop in $Y-J\cdot C$ is homotopically trivial in
$Y-C_{0}$.

Condition ($\ast$) is all that is needed in \cite{GG12} and \cite{W92} in
order to prove the main theorems. In \cite{GGM16} condition ($\ast$) is used
to show $Y$ is proper 2-equivalent to $T\times\mathbb{R}$ (where $T$ is a
tree). Interestingly, there are many examples of finitely presented groups $G
$ (and spaces) with infinite cyclic subgroups satisfying $(\ast)$ but the
fundamental group at $\infty$ of $G$ is not pro-monomorphic (see
\cite{GGM16}). In fact, if $G$ has pro-monomorphic fundamental group at
$\infty$, then either $G$ is simply connected at $\infty$ or (by a result of
B. Bowditch \cite{Bo04}) $G$ is virtually a closed surface group and $\pi
_{1}^{\infty}(G)=\mathbb{Z}$.

Our co-semistability definition generalizes the conditions of $(\ast)$ in two
fundamental ways and our main theorem still concludes that $Y$ has semistable
fundamental group at $\infty$ (just as in the main theorem of \cite{GG12}).

1) First we expand $J$ from an infinite cyclic group to an arbitrary finitely
generated group and we allow $J$ to change as compact subsets of $Y$ become larger.

2) We weaken the requirement that loops in $Y-J\cdot C$ be trivial in
$Y-C_{0}$ to only requiring that loops in $Y-J\cdot C$ can be ``pushed"
arbitrarily far out in $Y-C_{0}$.

We are now ready to set up our co-semistability definition. A subset $S$ of
$Y$ is \textit{bounded} in $Y$ if $S$ is contained in a compact subset of $Y$.
Otherwise $S$ is \textit{unbounded} in $Y$. Fix an infinite finitely generated
group $J$ acting as covering transformations on $Y$ and a finite generating
set $J^{0}$ of $J$. Assume $J$ respects a cell structure on $Y$. Let
$p:Y\rightarrow J\backslash Y$ be the quotient map. If $K$ is a subset of $Y$,
and there is a compact subset $D$ of $Y$ such that $K\subset J\cdot D$
(equivalently $p(K)$ has image in a compact set), then $K$ is a $J$%
-\textit{bounded} subset of $Y$. Otherwise $K$ is a $J$-\textit{unbounded}
subset of $Y$. If $r:[0,\infty)\rightarrow Y$ is proper and $pr$ has image in
a compact subset of $J\backslash Y$ then $r$ is said to be $J$%
-\textit{bounded}. Equivalently, $r$ is a $J$-bounded proper edge path in $Y$
if and only if $r$ has image in $J\cdot D$ for some compact set $D\subset Y$.
In this case, there is an integer $M$ (depending only on $D$) such that each
vertex of $r$ is within (edge path distance) $M$ of a vertex of $J\ast$. Hence
$r$ `determines' a unique end of the Cayley graph $\Lambda(J,J^{0})$.

For a non-empty compact set $C_{0}\subset Y$ and finite subcomplex $C$
containing $C_{0}$ in $Y$, let $U$ be a $J$-unbounded component of $Y-J\cdot
C$ and let $r$ be a $J$-bounded proper ray with image in $U$. We say $J$ is
co-\textit{semistable at $\infty$ in }$U$ \textit{with respect to $r$ and}
$C_{0}$ if for any compact set $D$ and loop $\alpha:[0,1]\rightarrow U$ with
$\alpha(0)=\alpha(1)=r(0)$ there is a homotopy $H:[0,1]\times\lbrack
0,n]\rightarrow Y-C_{0}$ such that $H(t,0)=\alpha(t)$ for all $t\in
\lbrack0,1]$ and $H(0,s)=H(1,s)=r(s)$ for all $s\in\lbrack0,n]$ and
$H(t,n)\subset Y-D$ for all $t\in\lbrack0,1]$. This means that $\alpha$ can be
pushed along $r$ by a homotopy in $Y-C_{0}$ to a loop in $Y-D$. We say $J$ is
\textit{co-semistable at $\infty$ in }$Y$ \textit{with respect to $C_{0}$}
(and $C$) if $J$ is co-semistable at $\infty$ in $U$ with respect to $r$ and
$C_{0}$ for each $J$-unbounded component $U$ of $Y-J\cdot C$, and any proper
$J$-bounded ray $r$ in $U$. Note that if $\hat{C}$ is a finite complex
containing $C$, then $J$ is also co-semistable at $\infty$ in $Y$ with respect
to $C_{0}$ and $\hat{C}$.

It is important to notice that our definition only requires that loops in $U$
can be pushed arbitrarily far out in $Y-C_{0}$ along proper $J$-bounded rays
in $U$ (as opposed to all proper rays in $U$).

\section{An outline of the proof of the main theorem}

\label{out} A number of technical results are necessary to prove the main
theorem. The outline in this section is intended to give the geometric
intuition behind these results and describe how they connect to prove the main
theorem. Figure 6 will be referenced throughout this section. Here
$C_{0}$ is an arbitrary compact subset of $Y$, $J^{0}$ is a finite generating
set for the group $J$ which respects a locally finite cell structure on $Y$
and acts as covering transformations on $Y$. The finite subcomplex $C$ of $Y$
is such that $J$ is co-semistable at $\infty$ in $Y$ with respect to $C_{0}$
and $C$, and $J$ is semistable at $\infty$ in $Y$ with respect to $J^{0}$,
$C_{0}$ and $C$. The proper base ray is $r_{0}$, $E$ is a finite union of
specially selected compact sets and $\alpha$ is a loop based on $r_{0}$ with
image in $Y-E$. The path $\alpha$ is broken into subpaths $\alpha=(\alpha
_{1},e_{1},\beta_{1},\tilde{e}_{1},\alpha_{2},\ldots,\alpha_{n})$ where the
$\alpha_{i}$ lie in $J\cdot C$, the $\beta_{i}$ lie in $Y-J\cdot C$ and the
edges $e_{i}$ and $\tilde{e}_{i}$ serve as \textquotedblleft transition
edges". We let $F$ be an arbitrary large compact set and we must show that
$\alpha$ can be pushed along $r_{0}$ to a loop outside of $F$ by a homotopy
avoiding $C_{0}$ (see Theorem \ref{ssequiv} (2)).

In $\S $\ref{stabil} and $\S $\ref{SVK} we show $Y-J\cdot C$ has only finitely
many $J$-unbounded components (up to translation in $J$) and that the
stabilizer of any one of these components is infinite and finitely generated.
We pick a finite collection of $J$-unbounded components of $Y-J\cdot C$ such
that no two are $J$-translates of one another, and any $J$-unbounded component
of $Y-J\cdot C$ is a translate of one of these finitely many. Each
$g_{i}U_{f(i)}$ in Figure 6 is such that $g_{i}\in J$ and $U_{f(i)}$ is one of
these finitely many components. The edges $e_{i}$ have initial vertex in
$J\cdot C$ and terminal vertex in $g_{i}U_{f(i)}$. Similarly for $\tilde
e_{i}$. The fact that the stabilizer of a $J$-unbounded component of $Y-J\cdot
C$ is finitely generated and infinite allows us to construct the proper edge
path rays $r_{i}$, $\tilde r_{i}$, $s_{i}$ and $\tilde s_{i}$ in Figure 6. Let
$S_{i}$ be the (finitely generated infinite) $J$-stabilizer of $g_{i}U_{f(i)}%
$. Lemma \ref{JendsUnify} allows us to construct proper edge path rays $r_{i}$
in $J\cdot C$ (far from $C_{0}$) that are ``$S_{i}$-edge paths", and proper
rays $s_{i}$ in $g_{i}U_{f(i)}$ so that $s_{i}$ and $r_{i}$ are (uniformly
over all $i$) ``close" to one another. Hence $r_{i}$ is properly homotopic
$rel\{r_{i}(0)\}$ to $(\gamma_{i}, e_{i}, s_{i})$ by a homotopy in $Y-C_{0}$.
This mean $e_{i}$ can be ``pushed" between $s_{i}$ and $(\gamma_{i}^{-1},
r_{i})$ into $Y-F$ by a homotopy avoiding $C_{0}$ and we have the first step
in moving $\alpha$ into $Y-F$ by a homotopy avoiding $C_{0}$. Similarly for
$\tilde r_{i}$, $\tilde s_{i}$ and $\tilde e_{i}$.

Since all of the paths/rays $\alpha_{i}$, $\gamma_{i}$, $r_{i}$, $\tilde
\gamma_{i}$, and $\tilde r_{i}$ have image in $J\cdot C$, they are uniformly
(only depending on the size of the compact set $C$) close to $J$-paths/rays.
But the semistability at $\infty$ of $J$ in $Y$ with respect to $C_{0}$ then
implies there is a path $\delta_{i}$ connecting $(\tilde\gamma_{i-1}^{-1},
\tilde r_{i-1})$ and $(\alpha_{i}, \gamma_{i}^{-1}, r_{i})$ in $Y-F$ such that
the loop determined by $\delta_{i}$ and the initial segments of $(\tilde
\gamma_{i-1}^{-1}, \tilde r_{i-1})$ and $(\alpha_{i}, \gamma_{i}^{-1}, r_{i})$
is homotopically trivial by a homotopy avoiding $C_{0}$. Geometrically that
means $\alpha_{i}$ can be pushed outside of $F$ by a homotopy between
$(\tilde\gamma_{i-1}^{-1}, \tilde r_{i-1})$ and $(\gamma_{i}^{-1},r_{i})$, and
with image in $Y-C_{0}$.

All that remains is to push the $\beta_{i}$ into $Y-F$ by a homotopy between
$s_{i}$ and $\tilde{s}_{i}$. A serious technical issue occurs here. If we knew
that $s_{i}$ and $\tilde{s}_{i}$ converged to the same end of $g_{i}U_{f(i)}$
then we could find a path in $g_{i}U_{f(i)}-F$ connecting $s_{i}$ and
$\tilde{s}_{i}$ and Lemma \ref{USS} explains how to use the assumtion that $J$
is co-semistable at $\infty$ in $Y$ with respect to $C_{0}$, to slide
$\beta_{i}$ between $s_{i}$ and $\tilde{s}_{i}$ to a path in $Y-F$, finishing
the proof of the main theorem. But at this point there is no reason to believe
$s_{i}$ and $\tilde{s}_{i}$ determine the same end of $g_{i}U_{f(i)}$. This is
where two of the main lemmas (and two of the most important ideas) of the
paper, Lemmas \ref{far} and \ref{last} come in. All but finitely many of the
components $gU_{i}$ of $Y-J\cdot C$ avoid a certain compact subset of $E$. If
$g_{i}U_{f(i)}$ is one of these components then Lemma \ref{far} explains how
to select the proper ray $\tilde{r}_{i}$ and a path $\psi$ in $Y-F$ connecting
$r_{i}$ and $\tilde{r}_{i}$ so that the loop determined by $\psi$, initial
segments of $r_{i}$ and $\tilde{r}_{i}$ and the path $(\gamma_{i},e_{i}%
,\beta_{i},\tilde{e}_{i},\tilde{\gamma}_{i}^{-1})$ is homotopically trivial in
$Y-C_{0}$ (so that the section of $\alpha$ defined by $(e_{i},\beta_{i}%
,\tilde{e}_{i}$) can be pushed into $Y-F$ by a homotopy between $(\gamma
_{i}^{-1},r_{i})$ and $(\tilde{\gamma}_{i}^{-1},\tilde{r}_{i})$). Lemma
\ref{last} tells us how to select the compact set $E$ so that if
$g_{i}U_{f(i)}$ is one of the finitely many remaining components of $Y-J\cdot
U$, then the proper rays $s_{i}$ and $\tilde{s}_{i}$ can be selected, so that
$s_{i}$ and $\tilde{s}_{i}$ converge to the same end of $g_{i}U_{f(i)}$. In
either case, $\alpha$ is homotopic $rel\{r_{0}\}$ to a loop in $Y-F$ by a
homotopy in $Y-C_{0}$.

\section{Stabilizers of $J$-unbounded components}

\label{stabil} Throughout this section, $J$ is a finitely generated group
acting as cell preserving covering transformations on a simply connected
locally finite 1-ended CW complex $Y$ and $p:Y\to J\backslash Y$ is the
quotient map. Suppose $C$, is a large (see Theorem \ref{Craig}) finite
subcomplex of $Y$ and $U$ is a $J$-unbounded component of $Y-J\cdot C$. Lemma
\ref{trans} and Theorem \ref{Craig} show the $J$-stabilizer of $U$ is finitely
generated and infinite. Lemma \ref{Jends} shows that there is a finite
subcomplex $D(C)\subset Y$ such that for any compact $E$ containing $D $ and
any $J$-unbounded component $U$ of $Y-J\cdot C$ there is a special bijection
$\mathcal{M}$ between the set of ends of the $J$-stabilizer of $U$ and the
ends of $U\cap(J\cdot E)$. For $C$ compact in $Y$, Lemma \ref{finite} shows
there are only finitely many $J$-unbounded components of $Y-J\cdot C$ up to
translation in $J$.

Suppose that $J$ is semistable at $\infty$ in $Y$ with respect to $C_{0}$ and
$C$, $U$ is a $J$-unbounded component of $Y-J\cdot C$ and $J$ is co-semistable
at $\infty$ in $U$ with respect to the proper $J$-bounded ray $r$ and $C_{0}$.
Once again co-semistability at $\infty$ only depends on the 2-skeleton of $Y$
and from this point on we may assume that $Y$ is 2-dimensional. The next two
lemmas reduce complexity again by showing that in certain instances we need
only consider locally finite 2-complexes with edge path loop attaching maps on
2-cells. Such complexes are in fact simplicial and this is important for our
arguments in $\S $\ref{SVK}.

\begin{lemma}
\label{reduce} Suppose $Y$ is a locally finite 2-complex and the finitely
generated group $J$ acts as cell preserving covering transformations on $Y$,
then there is a $J$-equivariant subdivision of the 1-skeleton of $Y$ and a
locally finite 2-complex $X$ also admitting a cell preserving $J$-action such that:

\begin{enumerate}
\item The image of a 2-cell attaching map for $Y$ is a finite subcomplex of
$Y$.

\item The space $X$ has the same 1-skeleton as $Y$ and there is a
$J$-equivariant bijection between the cells of $Y$ and $X$ that is the
identity on vertices and edges and if $a$ is a 2-cell attaching map for $Y$
and $a^{\prime}$ is the corresponding 2-cell attaching map for $X$ then $a$
and $a^{\prime}$ are homotopic in the image of $a$, and $a^{\prime}$ is an
edge path loop with the same image as $a$.

\item The action of $J$ on $X$ is the obvious action induced by the action of
$J$ on $Y$.

\item If $K_{1}$ is a finite subcomplex of $Y$ and $K_{2}$ is the
corresponding finite subcomplex of $X$, then there is a bijective
correspondence between the $J$-unbounded components of $Y-J\cdot K_{1}$ and
$X-J\cdot K_{2}$, so that if $U_{1}$ is a $J$-unbounded component of $Y-J\cdot
K_{1}$ and $U_{2}$ is the corresponding component of $X-J\cdot K_{2}$ then
$U_{1}$ and $U_{2}$ are both a union of open cells, and the bijection of cells
between $Y$ and $X$ induces a bijection between the open cells of $U_{1}$ and
$U_{2}$. In particular, the $J$-stabilizer of $U_{1}$ is equal to that of
$U_{2}$.
\end{enumerate}
\end{lemma}

\begin{proof}
Suppose $D$ is a 2-cell of $Y$ and the attaching map on $S^{1}$ for $D$ is
$a_{D} $. Then the image of $a_{D}$ is a compact connected subset of the
1-skeleton of $Y$. If $e$ is an edge of $Y$ then $im(a_{D})\cap e$ is either
$\emptyset$, a single closed interval or a pair of closed intervals (we
consider a single point to be an interval). In any case add vertices when
necessary to make the end points of these intervals vertices. This process is
automatically $J$-equivariant and locally finite. The map $a_{D}$ is homotopic
(in the image of $a_{D}$) to an edge path loop $b_{D}$ with image the same as
that of $a_{D}$. Let $Z$ be the 1-skeleton of $Y$. Attach a 2-cell $D^{\prime
}$ to $Z$ with attaching map $b_{D}$. For $j\in J$ the attaching map for $jD$
is $ja_{D}$ and we automatically have an attach map for $X$ (corresponding to
the cell $jD$) defined by $jb_{D}$. This construction is $J$-equivariant. Call
the resulting locally finite 2-complex $X$ and define the action of $J$ on $X$
in the obvious way.

It remains to prove part 4. Suppose $K_{1}$ and $K_{2}$ are corresponding
finite subcomplexes of $Y$ and $X$ respectively. Recall that vertices are open
(and closed) cells of a CW complex and every point of a CW complex belongs to
a unique open cell. If $A$ is an open cell of $Y$ then either $A$ is a cell of
$J\cdot K_{1}$ or $A$ is a subset of $Y-J\cdot K_{1}$.

\medskip

\noindent\textbf{Claim \ref{reduce}.1} Suppose $U$ is a component of $Y-J\cdot
K_{1}$. If $p$ and $q$ are distinct points of $U$ then there is a sequence of
open cells $A_{0},\ldots, A_{n}$ of $U$ such that $p\in A_{0}$, $q\in A_{n}$
and either $A_{i}\cap\bar A_{i+1}\ne\emptyset$ or $\bar A_{i}\cap A_{i+1}%
\ne\emptyset$. (Here $\bar A$ is the closure of $A$ in $Y$, equivalently the
closed cell corresponding to $A$.)

\begin{proof}
Let $\alpha$ be a path in $U$ from $p$ to $q$. By local finiteness, there are
only finitely many closed cells $B_{0}, \ldots, B_{n}$ that intersect the
compact set $im(\alpha)$. Note that $B_{i}\not \subset K$ so that the open
cell $A_{i}$ for $B_{i}$ is a subset of $U$. In particular, $im(\alpha)\subset
A_{0}\cup\cdots\cup A_{n}$. Let $0=x_{0}$ and assume that $\alpha(x_{0})=p\in
A_{0}$. Let $x_{1}$ be the last point of $\alpha^{-1}(B_{0})$ in $[0,1]$ (it
may be that $x_{1}=x_{0}$). If $\alpha(x_{1})\not \in A_{0}$ then
$\alpha(x_{1})\in A_{1}\cup\cdots\cup A_{n}$ and assume that $\alpha(x_{1})\in
A_{1}$. In this case $\alpha(x_{1})\in\bar A_{0}\cap A_{1}(=B_{0}\cap A_{1})$.

If $\alpha(x_{1})\in A_{0}$, then take a sequence of points $\{t_{i}\}$ in
$(x_{1},1] $ converging to $x_{1}$. Infinitely many $\alpha(t_{i})$ belong to
some $A_{j}$ for $j\geq1$ (say $j=1$). Then $\alpha(x_{1})\in A_{0}\cap\bar
A_{1}$.

Let $x_{2}$ be the last point of $\alpha^{-1} (B_{1})$ in $[0,1]$. Continue inductively.
\end{proof}

\noindent\textbf{Claim \ref{reduce}.2} If $A_{1}\ne A_{2}$ are open cells of
$Y$ such that $A_{1}\cap\bar A_{2}\ne\emptyset$ and $A_{i}$ corresponds to the
open cell $Q_{i}$ of $X$ for $i\in\{1,2\}$, then $Q_{i}\cap\bar Q_{i+1}%
\ne\emptyset$.

\begin{proof}
We only need check this when $A_{1}$ or $A_{2}$ is a 2-cell (otherwise
$Q_{i}=A_{i}$). Note that $A_{1}$ is not a 2-cell, since otherwise $A_{1}%
\cap\bar A_{2}=\emptyset$. If $A_{2}$ is a 2-cell, and $A_{1}\cap\bar A_{2}%
\ne\emptyset$ then by construction $A_{1}\subset\bar A_{2}$, and $Q_{1}%
\subset\bar Q_{2}$.
\end{proof}

Write $U$ as a union $\cup_{i\in I}A_{i}$ of the open cells in $U$. Let
$Q_{i}$ be the open cell of $X$ corresponding to $A_{i}$. By Claims
\ref{reduce}.1 and \ref{reduce}.2, $\cup_{i\in I}Q_{i}$ is a connected subset
of $X-J\cdot K_{2}$. The roles of $X$ and $Y$ can be reversed in Claims
\ref{reduce}.1 and \ref{reduce}.2. Then writing a component of $X-J\cdot
K_{2}$ as a union of its open cells $\cup_{l\in L}Q_{l}$ (and letting $A_{l}$
be the open cell of $Y$ corresponding to $Q_{l}$) we have $\cup_{l\in L}A_{l}$
is a connected subset of $Y-J\cdot K_{1}$.
\end{proof}

\begin{remark}
\label{HE} There are maps $g:X\rightarrow Y$ and $f:Y\rightarrow X$ that are
the identity on 1-skeletons and such that $fg$ and $gf$ are properly homotopic
to the identity maps relative to the 1-skeleton. In particular, $X$ and $Y$
are proper homotopy equivalent. This basically follows from the proof of
Theorem 4.1.8 of \cite{G}. These facts are not used in this paper.
\end{remark}

The remainder of this section is a collection of elementary (but useful)
lemmas. The \textit{boundary} of a subset $S$ of $Y$ (denoted $\partial S$) is
the closure of $S$ (denoted $\bar S$) delete the interior of $S$. If $K$ is a
subcomplex of a 2-complex $Y$ then $\partial K$ is a union of vertices and edges.

\begin{lemma}
\label{cover} If $A\subset Y$, then $p(A)=p(J\cdot A)$ and $p^{-1}
(p(A))=J\cdot A$. If $C$ is compact in $Y$ and $B$ is compact in $J\backslash
Y$ such that $p(C)\subset B$, then there is a compact set $A\subset Y$ such
that $C\subset A$ and $p(A)=B$.
\end{lemma}

\begin{proof}
The first part of the lemma follows directly from the definition of $J\cdot A
$. Cover $B\subset J\backslash Y$ by finitely many evenly covered open sets
$U_{i}$ for $i\in\{1,\ldots, n\}$ such that $\bar U_{i}$ is compact and evenly
covered. Pick a finite number of sheets over the $\bar U_{i}$ that cover $C$
and so that there is at least one sheet over each $\bar U_{i}$. Call these
sheets $K_{1},\ldots, K_{m}$. Let $A=(\cup_{i=1}^{m} K_{i})\cap p^{-1}(B)$.
Then $C\subset A$, and $A$ is compact since $(\cup_{i=1}^{m} K_{i})$ is
compact and $p^{-1}(B)$ is closed. We claim that $p(A)=B$. Clearly
$p(A)\subset B$. If $b\in B$, then there is $j\in\{1,\ldots, n\}$ such that
$b\in\bar U_{j}$. Then there is $k_{b}\in K_{j^{\prime}}$ such that
$p(k_{b})=b$, and so $k_{b}\in p^{-1}(B)\cap(\cup_{i=1}^{m} K_{i})$ and $p$
maps $A$ onto $B$.
\end{proof}

\begin{remark}
\label{R2} If $C$ is a compact subset of $Y$, $j$ is an element of $J$ and $U
$ is a component of $Y-J\cdot C$ then $j(U)$ is a component of $Y-J\cdot C$,
and $p(U)$ is a component of $J\backslash Y-p(C)$.
\end{remark}

\begin{lemma}
\label{Unbd} Suppose $C$ is a non-empty compact subset of $Y$ and $U$ is an
unbounded component of $Y-J\cdot C$. Then $\partial U$ is an unbounded subset
of $J\cdot C$.
\end{lemma}

\begin{proof}
Otherwise $\partial U$ is closed and bounded in $Y$ and therefore compact. But
$\partial U$ separates $U$ from $J\cdot C$, contradicting the fact that $Y$ is 1-ended.
\end{proof}

The next remark establishes a minimal set of topological conditions on a
topological space $X$ in order to define the number of ends of $X$.

\begin{remark}
\label{ends} If $X$ is a connected, locally compact, locally connected
Hausdorff space and $C$ is compact in $X$, then $C$ union all bounded
components of $X-C$ is compact, any neighborhood of $C$ contains all but
finitely many components of $X-C$, and $X-C$ has only finitely many unbounded components.
\end{remark}

\begin{lemma}
\label{finite} Suppose $C$ is a compact subset of $Y$ and $U$ is a component
of $Y-J\cdot C$. Then $U$ is $J$-unbounded if and only if $p(U)$ is an
unbounded component of $J\backslash Y-p(C)$. Hence up to translation by $J$
there are only finitely many $J$-unbounded components of $Y-J\cdot C$.
\end{lemma}

\begin{proof}
First observe that $p(C)\cap p(U)=\emptyset$. Suppose $p(U)$ is unbounded.
Choose a ray $r:[0,\infty)\to p(U)$ such that $r$ is proper in $J\backslash Y
$. Select $u\in U$ such that $p(u)=r(0)$. Lift $r$ to $\tilde r$ at $u$. Then
$\tilde r$ has image in $U$, and there is no compact set $D\subset Y$ such
that $im (\tilde r)\subset J\cdot D$. Hence $U$ is $J$-unbounded. If $U$ is
$J$-unbounded then by definition, $p(U)$ is not a subset of a compact subset
of $Y$.
\end{proof}

\begin{lemma}
\label{Bded} Suppose $C$ is a compact subset of $Y$. Then there is a compact
subset $D\subset Y$ such that $C\subset D$, every $J$-bounded component of
$Y-J\cdot C$ is a subset of $J\cdot D$ and each component of $Y-J\cdot D$ is
$J$-unbounded.
\end{lemma}

\begin{proof}
Let $U$ be a $J$-bounded component of $Y-J\cdot C$. Then $p(U)$ is a bounded
component of $J\backslash Y-p(C)$. Let $B$ be the union of $p(C)$ and all
bounded components of $J\backslash Y-p(C)$. Then $B$ is compact (Remark
\ref{ends}). By Lemma \ref{cover}, there is a compact set $D$ containing $C$
such that $p(D)=B$.
\end{proof}

\begin{lemma}
\label{Int} Suppose $C$ and $D$ are finite subcomplexes of $Y$. Then only
finitely many $J$-unbounded components of $Y-J\cdot C$ intersect $D$.
\end{lemma}

\begin{proof}
Note that $J\cdot C$ is a subcomplex of $Y$. If the lemma is false, then for
each $i\in\mathbb{Z}^{+}$ there are distinct unbounded components $U_{i}$ of
$Y-J\cdot C$ such that $U_{i}\cap D\ne\emptyset$. Choose $u_{i}\in U_{i}\cap
D$. Let $E_{i}$ be an (open) cell containing $u_{i}$. Then $E_{i}\subset
U_{i}$ and the $E_{i}$ are distinct. Then infinitely many cells of $Y$
intersect $D$, contrary to the local finiteness of $Y$.
\end{proof}

\begin{lemma}
\label{trans} Suppose $C$ is a finite subcomplex of $Y$ and $U$ is a
$J$-unbounded component of $Y-J\cdot C$. Then there are infinitely many $j\in
J$ such that $j(U)=U$. In particular the $J$-stabilizer of $U$ is an infinite
subgroup of $J$.
\end{lemma}

\begin{proof}
If $x\in\partial U\subset\partial(J\cdot C)$ then any neighborhood of $x$
intersects $U$. Let $x_{1},x_{2},\ldots$ be sequence in $U$ converging to $x$.
By local finiteness infinitely many $x_{i}$ belong to some open cell $D$ of
$U$ and so $x\in\bar D$. By Lemma \ref{Unbd}, there are infinitely many open
cells $D$ of $U$ and distinct $j_{D}\in J$ such that $j_{D}(C)\cap\bar
D\ne\emptyset$. For all such $D$, $j_{D}^{-1}(\bar D)\cap C\ne\emptyset$ and
by the local finiteness of $Y$, there are infinitely many such $D$ with
$j_{D}^{-1}(D)$ all the same. If $j_{D_{1}}^{-1}(D_{1})=j_{D_{2}}^{-1}(D_{2})$
then $j_{D_{2}}j_{D_{1}}^{-1}(D_{1})=D_{2}$ so $j_{D_{2}}j_{D_{1}}^{-1}$
stabilizes $U$.
\end{proof}

\begin{lemma}
\label{conj} Suppose $C$ is a finite subcomplex of $Y$, $U$ is a $J$-unbounded
component of $Y-J\cdot C$ and $S<J$ is the subgroup of $J$ that stabilizes
$U$. Then for any $g\in J$, the stabilizer of $gU$ is $gSg^{-1}$.
\end{lemma}

\begin{proof}
Simply observe that $hgU=gU$ if and only if $g^{-1}hgU=U$ if and only if
$g^{-1}hg\in S$ if and only if $h\in gSg^{-1}$.
\end{proof}

\begin{lemma}
\label{subcomp} Suppose $C\subset Y$ is compact and $R_{1}$ is a $J$-unbounded
component of $Y-J\cdot C$. If $D\subset Y$ is compact, and $C\subset D$ then
there is a $J$-unbounded component $R_{2}$ of $Y-J\cdot D$ such that
$R_{2}\subset R_{1}$.
\end{lemma}

\begin{proof}
Choose an unbounded component $V_{2}$ of $J\backslash Y-p(D)$ such that
$V_{2}\subset p(R_{1})$. By Lemma \ref{finite}, there is a component
$R_{2}^{\prime} $ of $Y-J\cdot D$ such that $p(R_{2}^{\prime})=V_{2}$ and so
$R_{2}^{\prime}$ is $J $-unbounded. Choose points $x\in R_{1}$ and $y\in
R_{2}^{\prime}$ such that $p(x)=p(y)\in V_{2}$. Then the covering
transformation taking $y$ to $x$ takes $R_{2}^{\prime}$ to a $J$-unbounded
component $R_{2}$ of $Y-J\cdot D$. As $x\in R_{2}\cap R_{1}$, we have
$R_{2}\subset R_{1}$.
\end{proof}

\section{Finite generation of stabilizers}

\label{SVK}

The following principal result of this section allows us to construct proper rays in $J$-unbounded components of $Y-J\cdot D$ that track corresponding proper rays in a copy of a Cayley graph of the corresponding stabilizer of that component. These geometric constructions are critical to the proof of our main theorem. 

\begin{theorem}
\label{Craig} Suppose $J$ is a finitely generated group acting as cell
preserving covering transformations on the simply connected, 1-ended,
2-dimensional, locally finite CW complex $Y$. Let $p:Y\to J\backslash Y$ be
the quotient map. Suppose $D$ is a connected finite subcomplex of $Y$ such
that the image of $\pi_{1}(p(D))$ in $\pi_{1}(J\backslash Y)$ (under the map
induced by inclusion of $p(D)$ into $J\backslash Y$) generates $\pi
_{1}(J\backslash Y)$. Then for any $J$-unbounded component $V$ of $Y-J\cdot
D$, the stabilizer of $V$ under the action of $J$ is finitely generated.
\end{theorem}

By Lemma \ref{reduce} and Remark \ref{HE} we may assume that $Y$ is
simplicial. Theorem 6.2.11\cite{G} is a cellular version of van Kampen's
theorem. The following is an application of that theorem.

\begin{theorem}
\label{VKP} Suppose $X_{1}$ and $X_{2}$ are path connected subcomplexes of a
path connected CW complex $X$, such that $X_{1}\cup X_{2}$=$X$, and $X_{1}\cap
X_{2}=X_{0}$ is non-empty and path connected. Let $x_{0}\in X_{0}$. For
$i=0,1,2$ let $A_{i}$ be the image of $\pi_{1}(X_{i}, x_{0})$ in $\pi
_{1}(X,x_{0})$ under the map induced by inclusion of $X_{i}$ into $X$. Then
$\pi_{1}(X,x_{0})$ is isomorphic to the amalgamated product $A_{1}\ast_{A_{0}}
A_{2}$.
\end{theorem}

\begin{theorem}
\label{FG} Suppose that $X$ is a connected locally finite 2-dimensional
simplicial complex. If $K$ is a finite subcomplex of $X$ such that the
inclusion map $i:K\hookrightarrow X$ induces an epimorphism on fundamental
group and $U$ is an unbounded component of $X-K$ then the image of $\pi
_{1}(U)$ in $\pi_{1}(X)$, under the map induced by the inclusion of $U$ into
$X$ is a finitely generated group.
\end{theorem}

\begin{proof}
If $V$ is a bounded component of $X-K$ then $V\cup K$ is a finite subcomplex
of $X$. So without loss, assume that each component of $X-K$ is unbounded. If
$e$ is edge in $X-K$ and both vertices of $e$ belong to $K$, then by
baracentric subdivision, we may assume that each open edge in $X-K$ has at
least one vertex in $X-K$. Equivalently, if both vertices of an edge belong to
$K$, then the edge belongs to $K$. If $T$ is a triangle of $X$ and each vertex
of $T$ belongs to $K$, then each edge belongs to $K$, and $T$ belongs to $K$
(otherwise the open triangle of $T$ is a bounded component of $X-K$).

The \textit{largest} subcomplex $Z$ of $X$ contained in a component $U$ of
$X-K$ contains all vertices of $X$ that are in $U$, all edges each of whose
vertices are in $U$, and all triangles each of whose vertices are in $U$.

\begin{lemma}
\label{SDR} Suppose that $U$ is a component of $X-K$ and $Z$ is the largest
subcomplex of $X$ contained in $U$. Then $Z$ is a strong deformation retract
of $U$. In particular, $Z$ is connected.
\end{lemma}

\begin{proof}
If $e$ (resp. $T$) is an open edge (resp. triangle) of $X$ that is a subset of
$U$, but not of $Z$, then some vertex of $e$ (respectively $T$) belongs to $K$
and some vertex of $e$ (resp. $T$) belongs to $Z$. Say $e$ has vertices $v$
and $w$ and $v\in Z$ and $w\in K$ then clearly $[v,w)$ linearly strong
deformation retracts to $v$. If $T$ is a triangle of $X$ with vertices $v,w\in
Z$ and $u\in K$ then for each point $p\in[v,w]$ the linear strong deformation
retraction from of $(u,p]$ to $p$ agrees with those defined for $(u,v]$ and
$(u,w]$ and defines a strong deformation for the triangle $[v,w,u]-\{u\}$ to
the edge $[v,w]$. Similarly if $v\in Z$ and $u, w\in K$. Combining these
deformation retractions gives a strong deformation retraction of $U$ to $Z$.
\end{proof}

Suppose that $U$ is a component of $X-K$ and $Z$ is the largest subcomplex of
$X$ contained in $U$. Let $Q_{1}$ be the (finite) subcomplex of $X$ consisting
of all edges and triangles that intersect both $U$ and $K$ (and hence
intersect both $Z$ and $K$). By Lemma \ref{SDR} we may add finitely many edges
in $Z$ to $Q_{1}$ so that the resulting complex $Q_{2}$, and $Q_{2}\cap Z$ are
connected. The complex $Q_{3}= Q_{2}\cup(X-U)$ is a connected subcomplex of
$X$.

The subcomplexes $Q_{3}$ and $Z$ are connected and cover $X$, and $Q_{3}\cap
Z=Q_{2}\cap Z$ is a non-empty connected finite subcomplex of $X$. Let $A_{0}$,
$A_{1}$ and $A_{2}$ be the image of $\pi_{1}(Q_{3}\cap Z)$, $\pi_{1}(Q_{3})$
and $\pi_{1}(Z) $ respectively in $\pi_{1}(X)$ under the homomorphism induced
by inclusion. By Theorem \ref{VKP}, $\pi_{1}(X)$ is isomorphic to the
amalgamated product $A_{1}\ast_{A_{0}}A_{2}$. Now as $K\subset Q_{3}$, $A_{1}
=\pi_{1}(X)$. But then normal forms in amalgamated products imply that
$A_{2}=A_{0}$. As $Q_{3}\cap Z$ is a finite complex, $A_{0}$ and hence $A_{2}$
is finitely generated. This completes the proof of Theorem \ref{FG}.
\end{proof}

Suppose $J$ is a finitely generated group acting on a simply connected
2-dimensional simplicial complex $Y$ and let $K$ be a finite subcomplex of
$J\backslash Y$ such that the image of $\pi_{1}(K)$ under the homomorphism
induced by the inclusion map of $K$ into $J\backslash Y$, generates $\pi
_{1}(J\backslash Y)$. Let $D$ be a finite subcomplex of $Y$ that projects onto
$K$ so that $p^{-1}(K)=J\cdot D$. Let $X_{1}$ be an unbounded component of
$J\backslash Y-K$. The number of $J$-unbounded components of $Y-J\cdot D$ that
project to $X_{1}$ is the index of the image of $\pi_{1}(X_{1})$ in $\pi
_{1}(J\backslash Y)=J$ under the homomorphism induced by inclusion; and the
stabilizer of such a $J$-unbounded component is isomorphic to the image of
$\pi_{1}(X_{1})$ in $\pi_{1}(J\backslash Y)=J$ under the homomorphism induced
by inclusion. Hence Theorem \ref{Craig} is a direct corollary of Theorem
\ref{FG}.

\section{A bijection between $J$-bounded ends and stabilizers}

\label{bij} As usual $J^{0}$ is a finite generating set for an infinite group
$J$ which acts as covering transformations on a 1-ended simply connected
locally finite 2-dimensional CW complex $Y$. Assume that $C$ is a finite
subcomplex of $Y$ and $U$ is a $J$-unbounded component of $Y-J\cdot C$. The
main result of this section connects the ends of the $J$-stabilizer of $U$ to
the $J$-bounded ends of $U$ (and allows us to construct the $r$ and $s$ rays
in Figure 6). Recall $z:(\Lambda(J,J^{0}),1)\to(Y,\ast)$ and $K$ is an integer
such that for each edge $e$ of $\Lambda$, $z(e)$ is an edge path of length
$\leq K$.

\begin{lemma}
\label{Jends2} Suppose $C$ and $D$ are finite subcomplexes of $Y$, $U$ is a
$J$-unbounded component of $Y-J\cdot C$ and some vertex of $J\cdot D$ belongs
to $U$. Let $S$ be the $J$-stabilizer of $U$. Then there is an integer
$N_{\ref{Jends2}}(U,C,D)$ such that for each vertex $v\in U\cap(J\cdot D)$
there is an edge path of length $\leq N$ from $v$ to $S\ast$ and for each
element $s\in S$ there is an edge path of length $\leq N$ from $s\ast$ to a
vertex of $U\cap(J\cdot D)$.
\end{lemma}

\begin{proof}
Without loss, assume that $\ast\in D$ and $D$ is connected. Let $A$ be an
integer such that any two vertices in $D$ can be connected by an edge path of
length $\leq A$. For each vertex $v$ of $U\cap(J\cdot D)$ let $\alpha_{v}$ be
a path of length $\leq A$ from $v$ to a vertex $w_{v}\ast$ of $J\ast$. The
covering transformation $w_{v}^{-1}$ takes $\alpha_{v}$ to an edge path ending
at $\ast$ and of length $\leq A$. The vertices of $U\cap(J\cdot D)$ are
partitioned into a finite collection of equivalence classes, where $v$ and $u
$ are related if $w_{v}^{-1}(\alpha_{v})$ and $w_{u}^{-1}(\alpha_{u})$ have
the same initial point. Equivalently, $w_{v}w_{u}^{-1}u=v$. In particular,
$u\sim v$ implies $w_{v}w_{u}^{-1}\in S$. Let $d_{\Lambda}$ denote edge path
distance in the Cayley graph $\Lambda(J, J^{0})$ and $|g|_{\Lambda}%
=d_{\Lambda} (1,g)$. Note that, as vertices of $\Lambda$:
\[
d_{\Lambda}(w_{v}w_{u}^{-1}, w_{v}) = |w_{u}|_{\Lambda}%
\]
For each (of the finitely many) equivalence class of vertices in $U\cap(J\cdot
D)$, distinguish $u$ in that class. Let $N_{1}$ be the largest of the numbers
$|w_{u}|_{\Lambda}$ (over the distinguished $u$). If $u$ is distinguished and
$v\sim u$ then let $\beta$ be an edge path in $\Lambda$ of length $\leq N_{1}$
from $w_{v}$ to $w_{v}w_{u}^{-1}$. Then $z\beta$ (from $w_{v}\ast$ to
$w_{v}w_{u}^{-1}\ast\in S\ast$) has length $\leq KN_{1}$. The path
$(\alpha_{v}, z\beta)$ (from $v$ to $w_{v}w_{u}^{-1}\ast\in S\ast$) has length
$\leq N_{1}K+A$.

Let $\alpha$ be an edge path from $\ast$ to a vertex of $U\cap(J\cdot D)$.
Then for each $s\in S$, $s(\alpha)$ is an edge path from $s\ast$ to a vertex
of $U\cap(J\cdot D)$. Let $N_{2}=|\alpha|$ then let $N$ be the largest of the
integers $N_{1}K+A$ and $N_{2}$.
\end{proof}

\begin{remark}
Assume we are in the setup of Lemma \ref{Jends2}. Suppose $g\in J$. Then each
vertex of $(gU)\cap(J\cdot D)$ is within $N$ of a vertex of $gS\ast$ and
within $N+|g|K$ of $gSg^{-1}\ast$ (as $d_{\Lambda}(gs, gsg^{-1})=|g^{-1}|$),
where by Lemma \ref{conj}, $gSg^{-1}$ stabilizes $gU$. Also, each vertex of
$gS\ast$ is within $N$ of a vertex of $(gU)\cap(J\cdot D)$ and each vertex of
$gSg^{-1}\ast$ is within $N+|g|K$ of a vertex of $(gU)\cap(J\cdot D)$. By
Lemma \ref{finite} there are only finitely many $J$-unbounded components of
$Y-J\cdot C$ up to translation in $J$. Hence finitely many integers $N$ cover
all cases.
\end{remark}

If $C\subset E$ are compact subsets of $Y$ and $U$ a $J$-unbounded component
of $Y-J\cdot C$, let $\mathcal{E}(U,E)$ be the set of equivalence classes of
$J$-bounded proper edge path rays of $U\cap(J\cdot E)$, where two such rays
$r$ and $s$ are equivalent if for any compact set $F$ in $Y$ there is an edge
path from a vertex of $r$ to a vertex of $s$ with image in $(U\cap(J\cdot
E))-F$. If $X$ is a connected locally finite CW complex, let $\mathcal{E}(X)$
be the set of ends of $X$. In the next lemma it is not necessary to factor the
map $m$ through $z:\Lambda(J,J^{0})\to Y$ in order to be true, but for our
purposes, it is more applicable this way. For a 2-dimensional CW complex $X$ and
subcomplex $A$ of $X$, let $A_{1}$ be the subcomplex comprised of $A$, union
all vertices connected by an edge to a vertex of $A$, union all edges with at
least one vertex in $A$. Let $St(A)$ be $A_{1}$ union all 2-cells whose
attaching maps have image in $A_{1}$. Inductively define $St^{n}%
(A)=St(St^{n-1}(A))$ for all $n>1$. The next lemma is a standard result that
we will employ a number of times.

\begin{lemma}
\label{kill} Suppose $L$ is a positive integer, then there is an integer
$M(L)$ such that if $\alpha$ is an edge path loop in $Y$ of length $\leq L$
and $\alpha$ contains a vertex of $J\ast$, then $\alpha$ is homotopically
trivial in $St^{M(L)}(v)$ for any vertex $v$ of $\alpha$.
\end{lemma}

\begin{proof}
Since $Y$ is simply connected each of the (finitely many) edge path loops at
$\ast$ which have length $\leq L$ is homotopically trivial in $St^{M_{1}}%
(\ast) $ for some integer $M_{1}$. If $\alpha$ is a loop at $\ast$ of length
$L$ and $v$ is a vertex of $\alpha$ then $St^{M_{1}}(\ast)\subset St^{M_{1}%
+L}(v)$ and so $\alpha$ is homotopically trivial in $St^{M}(v)$ where
$M=M_{1}+L$. The lemma follows by translation in $J$.
\end{proof}

\begin{lemma}
\label{Jends} Suppose $C$ is a finite subcomplex of $Y$ and $U$ is a
$J$-unbounded component of $Y-J\cdot C$. Let $S^{0}$ be a finite generating
set for $S$ (the $J$-stabilizer of $U$), and let $\Lambda(S,S^{0})$ be the
Cayley graph of $S$ with respect to $S^{0}$. Let $m_{1}:\Lambda(S,S^{0}%
)\to\Lambda(J,J^{0}) $ be an $S$-equivariant map where $m_{1}(v)=v$ for each
vertex $v$ of $\Lambda(S,S^{0})$, and each edge of $\Lambda(S,S^{0})$ is
mapped to an edge path in $\Lambda(J,J^{0})$. Let $m=zm_{1}:\Lambda
(S,S^{0})\to Y$. Then there is a compact set $D_{\ref{Jends}}(C,U,
S^{0})\subset Y$ such that for any compact subset $E$ of $Y$ containing $D$,
there is a bijection
\[
\mathcal{M}_{U}:\mathcal{E}(\Lambda(S,S^{0}))\twoheadrightarrow\hspace{-.21in}
\rightarrowtail\mathcal{E}(U, E)=\mathcal{E}(U,D)
\]
and an integer $I_{\ref{Jends}}(U,C,D)$ such that if $q$ is a proper edge path
ray in $\Lambda(S,S^{0})$ and $\mathcal{M}([q])=[t]$ then there is a
$t^{\prime}\in[t] $ such that for each vertex $v$ of $m(q)$ there is an edge
path of length $\leq I$ from $v$ to a vertex of $t^{\prime}$ and if $w$ is a
vertex of $t^{\prime}$ then there is an edge path of length $\leq I$ from $w$
to a vertex of $m(q)$.
\end{lemma}

\begin{proof}
Throughout this proof $\Lambda=\Lambda(S,S^{0})$. We call the points
$m(S)(= S\ast) \subset Y$, the $S$-vertices of $Y$. There is an integer
$\mathbf{B(S^{0})}$ such that if $e$ is an edge of $\Lambda$ then the edge
path $m(e)$ has length $\leq B$. Fix $\alpha_{0}$ an edge path in $Y$ from
$\ast$ to a vertex of $u\in U$. If $[v,w]$ is an edge of $\Lambda$ then
$(v\alpha_{0}^{-1}, m(e), w\alpha_{0})$ is an edge path of length $\leq
B+2|\alpha_{0}|$ in $Y$ connecting $vu$ and $wu$ (the terminal points of
$v(\alpha_{0})$ and $w(\alpha_{0})$). Hence there is an integer $\mathbf{A}$
(depending only on the integer $B+2|\alpha_{0}|$) and an edge path of length
$\leq A$ in $U$ from the terminal point of $v(\alpha_{0})$ to the terminal
point of $w(\alpha_{0})$. Let $\mathbf{I}=|\alpha_{0}|+max\{A,B\}$. Let
$\mathbf{D_{1}}$ be a finite subcomplex of $Y$ containing $St^{A+B}(\ast)\cup
St(C)$. By Lemma \ref{Jends2} there is an integer $\mathbf{N}$ such that each
vertex of $(J\cdot D_{1})\cap U$ is connected by an edge path of length $\leq
N$ to a vertex of $S\ast$. There is an integer $\mathbf{Z}$ such that if $a$
and $b$ are vertices of $U$ which belong to an edge path in $Y$ of length
$\leq N+|\alpha_{0}|$, and this path contains a point of $J\ast$, then there
is an edge path of length $\leq Z$ in $U$ connecting $a$ and $b$. Let
$\mathbf{D}$ contain $D_{1}\cup St^{Z+N}(\ast)$.

Let $q$ be a proper edge path ray in $\Lambda$ with $q(0)=1$. Let the
consecutive $S$-vertices of $m(q)$ be $v_{0}=\ast, v_{1},v_{2},\ldots$. (So
the edge path distance in $Y$ between $v_{i}$ and $v_{i+1}$ is $\leq B$.) For
simplicity assume that $v_{i}$ is the element of $S$ that maps $\ast$ to
$v_{i}$. Then $v_{i}(\alpha_{0})$ is an edge path that ends in $U$. By the
definition of $D_{1}$, there is an edge path $\beta_{i}$ in $U\cap(J\cdot D)$
from the end point of $v_{i}(\alpha_{0})$ to the end point of $v_{i+1}%
(\alpha_{0})$ of length $\leq A$ (see the left hand side of Figure 2). For
each vertex $v$ of the proper edge path ray $\beta_{q}=(\beta_{0}%
,\beta_{1},\ldots)$ (in $U\cap(J\cdot D)$) there is an edge path of length
$\leq A+|\alpha_{0}|\leq I$ from $v$ to a vertex of $m(q)$. For each vertex
$w$ of $m(q)$ there is an edge path of length $\leq B+|\alpha_{0}|\leq I$ from
$w$ to a vertex of $\beta_{q}$. In particular, $\beta_{q}$ is a proper
$J$-bounded ray in $U$. If $p\in[q]\in\mathcal{E}(\Lambda(S,S^{0}))$ (with
$p(0)=1$) then $m(p)$ is of bounded distance from $\beta_{p}$. If $\delta_{i}$
is a sequence of edge paths in $\Lambda$ each beginning at a vertex of $q$ and
ending at a vertex of $p$, such that any compact subset intersects only
finitely many $\delta_{i}$, then the paths $m(\delta_{i})$ connect $m(q)$ to
$m(p)$ and (since $m$ is a proper map) any compact subset of $Y$ intersects
only finitely many $m(\delta_{i})$. The $m(\delta_{i})$ determine (using
translates of $\alpha_{0}$ as above) edge paths in $U\cap(J\cdot D)$
connecting $\beta_{q}$ and $\beta_{p}$ so that $[\beta_{p}]=[\beta_{q}]$ in
$\mathcal{E}(U,E)$ for any finite subcomplex $E$ of $Y$ which contains $D$.
This defines a map $\mathcal{M}:\mathcal{E}(\Lambda)\to\mathcal{E}(U, E)$
which satisfies the last condition of our lemma and it remains to show that
$\mathcal{M}$ is bijective.

Let $r$ be a proper edge path $J$-bounded ray in $U$. Then $r$ has image in
$J\cdot E$ for some finite subcomplex $E$ containing $D$. Let $v_{1}%
,v_{2},\ldots$ be the consecutive vertices of $r$. By Lemma \ref{Jends2} there
is an integer $N_{E}$ such that each $v_{i}$ is within $N_{E}$ of $S\ast$. Let
$\tau_{i}$ be a shortest edge path from $v_{i}$ to $S\ast$, so that $|\tau
_{i}|\leq N_{E}$. We may assume without loss that the image of $\tau_{i}$ is
in $J\cdot E$. Let $w_{i}\in S\ast$ be the terminal point of $\tau_{i}$. Let
$z_{i}$ be the first vertex of $\tau_{i}$ in $J\cdot D_{1}$. Then the segment
of $\tau_{i}$ from $z_{i}$ to $w_{i}$ has length $\leq N$. For each $i$ there
is an edge path in $Y$ of length $\leq2N_{E}+1$ connecting $w_{i}$ to
$w_{i+1}$. Hence there is a proper edge path ray $q(r)$ in $\Lambda$ such that
$m(q(r))$ contains each $w_{i}$. The proper edge path ray $\beta_{q(r)}$ has
image in $U\cap(J\cdot D_{1})$ and there is an edge path of length $\leq Z$ in
$U\cap(J\cdot D)$ from $z_{i}$ to a vertex of $\beta_{q(r)}$. Hence there is
an edge path in $U\cap(J\cdot E)$ of length $\leq Z+N_{E}$ from $v_{i}$ to a
vertex of $\beta_{q(r)}$ so that $[r]=[\beta_{q(r)}]$ in $\mathcal{E}(U,E)$.
In particular, $\mathcal{M}$ is onto.

Finally we show $\mathcal{M}$ is injective. Suppose $a$ and $b$ are distinct
proper edge path rays in $\Lambda$ with initial point $1$, such that
$[\beta_{a}]=[\beta_{b}]$ in $\mathcal{E}(U,E)$ for some $E$ containing $D$.
Let $\tau_{i}$ be a sequence of edge paths in $U\cap(J\cdot E)$ where each
begins at a vertex of $\beta_{a}$, ends at a vertex of $\beta_{b}$ and so that
only finitely many intersect any given compact set (a cofinal sequence). By
the construction of $\beta_{a}$ and $\beta_{b}$ we may assume the initial
point of $\tau_{i}$ is the end point of $v_{i}\alpha_{0}$ for $v_{i}$ a vertex
of $a$ in $\Lambda$ and the terminal point of $\tau_{i}$ is the end point of
$w_{i}\alpha_{0}$ for $w_{i}$ a vertex of $b$. By Lemma \ref{Jends2} there is
an integer $N_{E}(\geq|\alpha_{0}|)$ such that each vertex of $\tau_{i}$ is
within $N_{E}$ of $S\ast$. For each $i$, this defines a finite sequence
$A_{i}$ of points in $S\ast$ beginning with $v_{i}\ast$ on $m(a)$, ending with
$w_{i}\ast$ on $m(b)$, each within $N_{E}$ of a point of $\tau_{i}$ and
adjacent points of $A_{i}$ are within $2N_{E}+1$ of one another. Since the
$\tau_{i}$ are cofinal, so are the $A_{i}$. Since the distance between
adjacent points of $A_{i}$ is bounded, if $u$ and $v$ are vertices of
$\Lambda(S,S^{0})$ such that $m(u)$ and $m(v)$ are adjacent in $A_{i}$ then
there is a bound on the distance between $u$ and $v$ in $\Lambda(S,S^{0})$.
This implies $a$ and $b$ determine the same end of $\Lambda(S,S^{0})$.
\end{proof}

\begin{remark}
\label{unify} Consider Lemma \ref{Jends} for components $gU$ of $Y-J\cdot C$
for $g\in J$. The stabilizer of $gU$ is $gSg^{-1}$ and there may be no bound
on the integers $I(gU, C,D)$ or the size of $D(C,gU)$. For $gU$, one can
consider instead $m_{g}:\Lambda(S,S^{0})\to Y$ by $m_{g}(x)=gm(x)$ (so
$m_{g}(1)=g\ast$). Lemma \ref{JendsUnify} is a generalization of Lemma
\ref{Jends} that applies to all $J$-translates of $U$. Since there are only
finitely many $J$-unbounded components of $Y-J\cdot C$ up to $J$-translation,
the dependency of $I$ and $D$ on $U$ can be eliminated and in the next lemma
$I_{\ref{JendsUnify}}$ and $D_{\ref{JendsUnify}}$ are taken to only depend on
$C$.
\end{remark}

For $C$ compact in $Y$, let $\mathcal{U}=\{U_{1},\ldots, U_{l}\}$ be a
set of $J$-unbounded components of $Y-J\cdot C$ such that if $U$ is any
$J$-unbounded component of $Y-J\cdot C$ then $U=gU_{i}$ for some $g\in J$ and
some $i\in\{1,\ldots, l\}$. Also assume that $U_{i}\ne gU_{j}$ for any $i\ne
j$ and any $g\in J$. Call $\mathcal{U}$ a \textit{component transversal} for
$Y-J\cdot C$. Let $S_{i}^{0}$ be a finite generating set for $S_{i}$, the
$J$-stabilizer of $U_{i}$ and $\Lambda_{i}=\Lambda(S_{i},S_{i}^{0})$ the
Cayley graph of $S_{i}$ with respect to $S_{i}^{0}$. For $g\in J$, let
$m_{(g,i)}:\Lambda_{i}\to Y$ be defined by $m_{(g,i)}(x)=gm_{i}(x)$ (where
$m_{i}:\Lambda_{i}\to Y$ is defined by Lemma \ref{Jends}). In particular,
$m_{(g,i)}(S_{i})=gS_{i}\ast$.

\begin{lemma}
\label{JendsUnify} For $i\in\{1,\ldots, l\}$, let  $D_{i}= D_{\ref{Jends}}(C,U_{i},S_{i}^{0})$,
$D_{\ref{JendsUnify}}(C)=\cup_{i=1}^{l}D_{i}\subset Y$, $I_{\ref{JendsUnify}%
}(C)=max\{I_{\ref{Jends}}(U_{i},C, D_{i}\}_{i=1}^{l}$ and $\mathcal{M}%
_{i}:\mathcal{E}(\Lambda_{i})\twoheadrightarrow\hspace{-.21in} \rightarrowtail
\mathcal{E}(U_{i}, E)$ (Lemma \ref{Jends}). For $E$ compact containing
$D_{\ref{JendsUnify}}(C)$ and $g\in J$, there is a bijection
\[
\mathcal{M}_{(g,i)}:\mathcal{E}(\Lambda_{i})\twoheadrightarrow\hspace{-.21in}
\rightarrowtail\mathcal{E}(gU_{i}, E) \hbox{ where } \mathcal{M}%
_{(g,i)}([q])=g\mathcal{M}_{i}([q])
\]
such that if $q$ is a proper edge path ray in $\Lambda_{i}$ and $\mathcal{M}%
_{(g,i)}([q])=[t]$ then there is $t^{\prime}\in[t] $ such that for each vertex
$v$ of $m_{(g,i)}(q)$ there is an edge path of length $\leq
I_{\ref{JendsUnify}}(C)$ from $v$ to a vertex of $t^{\prime}$ and if $w$ is a
vertex of $t^{\prime}$ then there is an edge path of length $\leq
I_{\ref{JendsUnify}}(C)$ from $w$ to a vertex of $m_{(g,i)}(q)=gm_{i}(q)$.
\end{lemma}

\vspace{.5in} \vbox to 2in{\vspace {-2in} \hspace {-1.3in}
\hspace{-.6 in}
\includegraphics[scale=1]{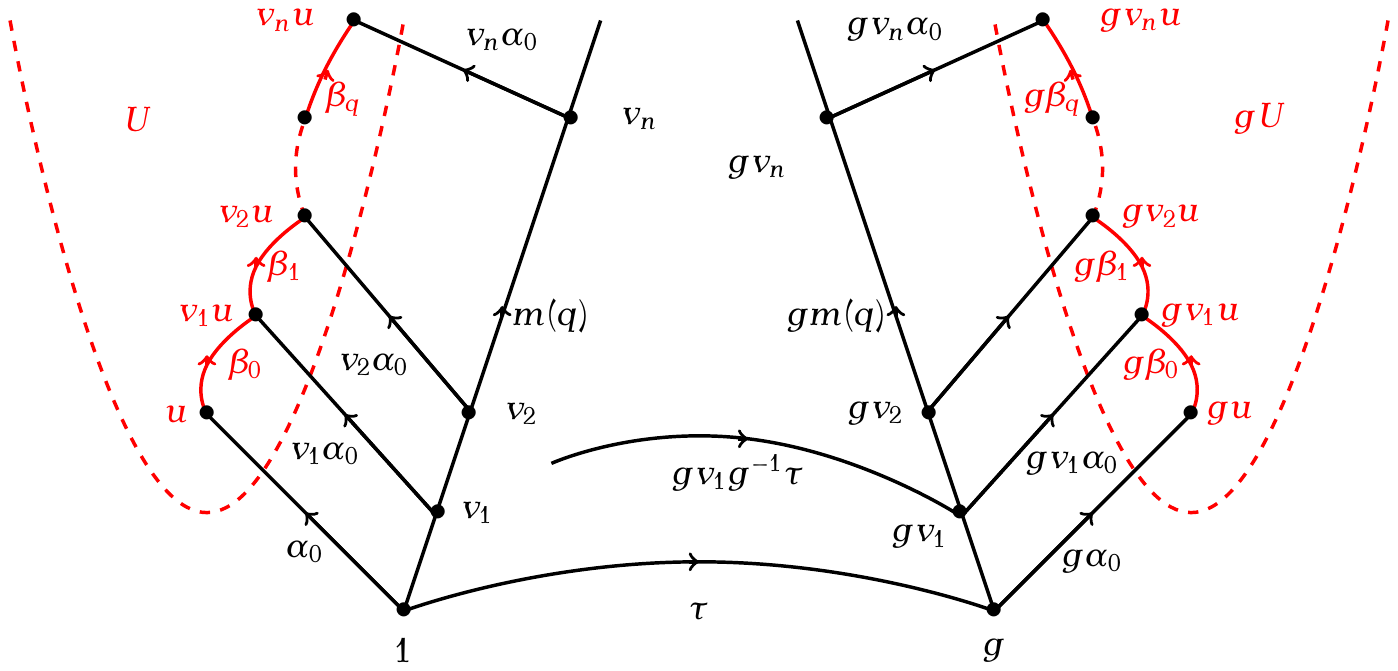}\label{FigEx2}
\vss }

\vspace{.5in}

\centerline{Figure 2}

\medskip

\section{Proof of the main theorem}

\label{proof}

We set notation for the proof of our main theorem. Let $C_{0}$ be compact in
$Y$, and $J^{0}$ be a finite generating set for the infinite group $J$ which
acts as cell preserving covering transformations on $Y$. Let $C$ be a finite
subcomplex of $Y$ such that $J$ is co-semistable at $\infty$ in $Y$ with
respect to $C_{0}$ and $C$, and $J$ is semistable at $\infty$ in $Y$ with
respect to $J^{0}$, $C_{0}$ and $C$. As in the setup for Lemma
\ref{JendsUnify} we let $\mathcal{U}=\{U_{1},\ldots,U_{l}\}$ be a component
transversal for $Y-J\cdot C$, $S_{i}^{0}$ be a finite generating set for
$S_{i}$, the $J$-stabilizer of $U_{i}$ and $\Lambda_{i}=\Lambda
(S_{i},S_{i}^{0})$ be the Cayley graph of $S_{i}$ with respect to $S_{i}^{0}$.
For $g\in J$, let $m_{(g,i)}:\Lambda_{i}\rightarrow Y$ be defined by
$m_{(g,i)}(x)=gm_{i}(x)$ (where $m_{i}:\Lambda_{i}\rightarrow Y$ is defined by
Lemma \ref{Jends}). In particular, $m_{(g,i)}(S_{i})=gS_{i}\ast$.

The next lemma is a direct consequence of Lemma \ref{Jends2}.

\begin{lemma}
\label{step1} Let $N_{i}$ be $N_{\ref{Jends2}}(U_{i},C,St(C))$ and
$N_{\ref{step1}}=max\{N_{1},\ldots, N_{l}\}$. If $g\in J$ and $[v,w]$ is an
edge of $Y$ with $v\in gU_{i}$ and $w\in J\cdot C$ then there are edge paths
of length $\leq N_{\ref{step1}}$ from $v$ and $w$ to $gS_{i}\ast$ and for each
$q\in S_{i}\ast$, an edge path of length $\leq N_{\ref{step1}}$ from $gq$ to a
vertex of $St(J\cdot C)\cap gU_{i}$.
\end{lemma}

\begin{lemma}
\label{step2} There is an integer $M_{\ref{step2}}(C)$ and compact set
$D_{\ref{step2}}(C)$ in $Y$ containing $St^{M_{\ref{step2}}} (C)$ such that
for any $U_{i}\in\{U_{1},\ldots, U_{l}\}$, $g\in J$ and edge $[v,w]$ of $Y$
with $v\in gU_{i}-D_{\ref{step2}}$ and $w\in J\cdot C$, (see Figure 3) we have
the following:

\begin{enumerate}
\item There is an edge path $\gamma$ of length $\leq N_{\ref{step1}}$ from a
vertex $x=gx^{\prime}\ast\in gS_{i}\ast$ to $w$, where $x^{\prime}$ is a
vertex in an unbounded component $Q$ of $\Lambda(S_{i},S^{0}_{i}%
)-m_{(g,i)}^{-1}(St^{M_{\ref{step2}}}(C))$.

\item If $\gamma$ is as in part 1, and $r_{0}^{\prime}$ is any proper edge
path ray in $Q$ beginning at $x^{\prime}$ (so $r_{0}= m_{(g,i)}%
(r_{0}^{\prime})$ is a proper edge path ray beginning at $x$), then there is a
proper $J$-bounded ray $s_{v}$ beginning at $v$ such that $s_{v}$ has image in
$gU_{i}$ and is properly homotopic $rel\{v\}$ to $([v,w],\gamma^{-1},r_{0})$
by a proper homotopy with image in $St^{M_{\ref{step2}}}(im(r_{0}))\subset
Y-C$. So (by hypothesis) $J$ is co-semistable at $\infty$ in $gU_{i}$ with
respect to $s_{v}$ and $C_{0}$.
\end{enumerate}
\end{lemma}

\begin{proof}
Let $A^{\prime}$ be an integer such that if $s\in\cup_{i=1}^{l}S_{i}^{0}$ then
there is an edge path of length $\leq A^{\prime}$ in $\Lambda(J,J^{0})$ from
$1 $ to $s$. The image of this path under $z:(\Lambda,1)\to(Y,\ast)$ is a path
in $Y$ of length $\leq KA^{\prime}=A$. Let $N= N_{\ref{step1}}$.
Select $B$ an integer such that if $a$ and $b$ are vertices of $St(J\cdot
C)\cap gU_{i}$ (for any $g\in J$ and $i\in\{1,\ldots, l\}$) of distance
$\leq2N+A+1$ in $Y$ then they can be joined by an edge path of length $\leq B
$ in $gU$. By Lemma \ref{kill} there is an integer $M_{\ref{step2}}$ such that
if $\beta$ is a loop in $Y$ of length $\leq A+B+2N+1$ and containing a vertex
of $J\ast$, then $\beta$ is homotopically trivial in $St^{M}(b)$ for any
vertex $b$ of $\beta$.

There are only finitely many pairs $(g,i)$ with $g\in J$ and $i\in\{1,\ldots,
l\}$ such that $gS_{i}\ast\cap St^{M}(C)\ne\emptyset$. If $gS_{i}\cap
St^{M}(C)=\emptyset$, then $m_{(g,i)}^{-1}(St^{M}(C))=\emptyset$. Lemma
\ref{step1} implies there is an edge path $\gamma$ of length $\leq
N_{\ref{step1}}$ from a vertex $x=gx^{\prime}\ast\in gS_{i}\ast$ to $w$. Now
let $r_{0}^{\prime}=(e_{0},e_{1},\ldots)$ be any proper edge path ray at
$x^{\prime}\in\Lambda(S_{i},S_{i}^{0})$. Let $\tau_{i}$ be the edge path
$m_{(g,i)}(e_{i})$ so that $\tau_{i}$ is an edge path in $Y$ of length $\leq
A$ and $r_{0}=m_{(g,i)}(r_{0}^{\prime})=(\tau_{1},\tau_{2},\ldots)$ is a
proper edge path at $x$ (see Figure 3). Let $x_{0}^{\prime}=x^{\prime}$ and $x_{j}^{\prime}$
be the end point of $e_{j}$ so that $x_{j}=gx_{j}^{\prime}\ast$ is the end
point of $\tau_{j}$. Let $\gamma_{0}=(\gamma,[w,v])$ (of length $\leq N+1$).
For $j\geq1$, let $\gamma_{j}$ be an edge path of length $\leq N_{\ref{step1}%
}$ from $x_{j}$ to $v_{j}\in gU_{i}\cap St(J\cdot C)$ (by Lemma \ref{step1}).
By the definition of $B $ there is an edge path $\beta_{j}$ in $gU_{i}$ from
$v_{j}$ to $v_{j+1}$ of length $\leq B$. Let $s_{v}$ be the proper edge path
$(\beta_{1},\beta_{2},\ldots)$, with initial vertex $v$. The loop
$(\gamma_{j-1}, \beta_{j}, \gamma_{j}^{-1}, \tau_{j}^{-1})$ has length $\leq
A+B+2N+1$ and contains the $J$-vertex $x_{j}$, and so is homotopically trivial
in $St^{M}(x_{j})\subset Y-C$. Combining these homotopies shows that $s_{v}$
is properly homotopic $rel\{v\}$ to $([v,w], \gamma^{-1}, r_{0})$ by a proper
homotopy with image in $St^{M}(im(r_{0}))\subset Y-C$. As long as
$D_{\ref{step2}}$ contains $St^{M}(C)$ the conclusion of our lemma is
satisfied for all such pairs $(g,i)$.

If $(g,i)$ is one of the finitely many pairs such that $gS_{i}\cap
St^{M}(C)\ne\emptyset$ then we need only find a compact $D_{(g,i)}$ so that
the lemma is valid for the pair $(g,i)$ and $D_{(g,i)}$, since we can let $D$
be compact containing $St^{M}(C)$ and the union of these finitely many
$D_{(g,i)}$.

\vspace{.5in} \vbox to 2in{\vspace {-2in} \hspace {-.5in}
\hspace{-.6 in}
\includegraphics[scale=1]{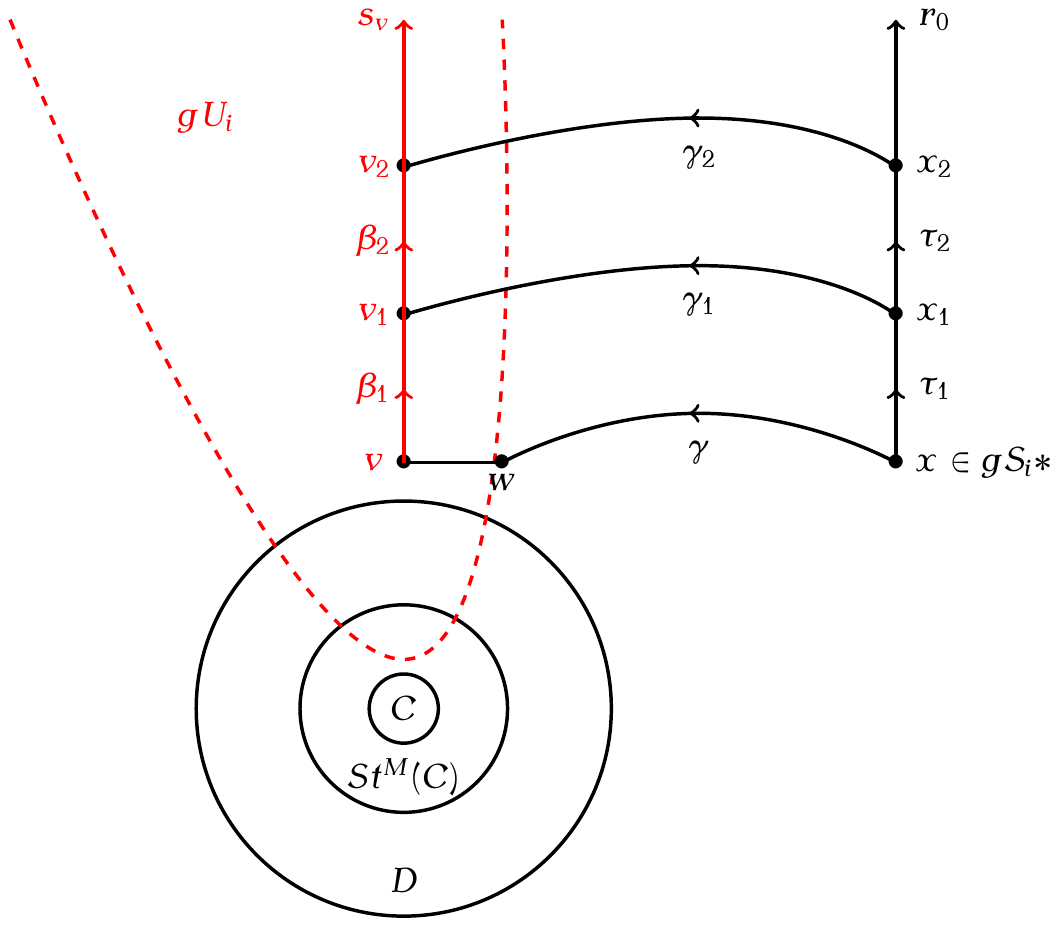}
\vss }

\vspace{1.5in}

\centerline{Figure 3}

\medskip

\medskip

Fix $(g,i)$ and let $E$ be compact in $\Lambda(S_{i},S_{i}^{0})=\Lambda_{i}$
containing the compact set $m_{(g,i)}^{-1}(St^{M}(C))$ and all bounded
components of $\Lambda_{i}-m_{(g,i)}^{-1}(St^{M}(C))$. Let $D_{(g,i)}$ be
compact in $Y$ containing $m_{(g,i)}(E)$. Select $\gamma$ exactly as in the
first case. Since $x^{\prime}$ is a vertex of $\Lambda_{i}$ in an unbounded
component $Q$ of $\Lambda_{i}-m_{(g,i)}^{-1}(St^{M}(C))$, there is a proper
edge path ray $r_{0}^{\prime}$ at $x^{\prime}$ with image in $Q$. Then
$r_{0}= m_{(g,i)}(r_{0}^{\prime})$ is a proper edge path ray at $x$ and
the vertices of $r_{0}^{\prime}$ are mapped to vertices $x_{0}=x, x_{1}%
,\ldots$ of $(gS_{i}\ast)-St^{M}(C)$. Select paths $\tau_{i}$ and $\beta_{i}$
as in the first case and the same argument shows that $s_{v}=(\beta_{1}%
,\beta_{2},\ldots)$ is properly homotopic $rel\{v\}$ to $([v,w], \gamma^{-1},
r_{0})$ by a proper homotopy with image in $St^{M}(im(r_{0}))\subset Y-C$.
\end{proof}

\begin{remark}
\label{ladder} The homotopy of Lemma \ref{step2} (pictured in Figure 3) of
$s_{v}$ to $([v,w],\gamma^{-1}, r_{0})$ is sometimes called a \textit{ladder
homotopy}. The \textit{rungs} of the ladder are the $\gamma_{i}$ and the
\textit{sides} of the ladder are $s_{v}$ and $r_{0}$. The loops determined by
two consecutive rungs and the segments of the two sides connecting these rungs
have bounded length and contain a vertex of $J\ast$. Lemma \ref{kill} implies
there is an integer $M$ such that each such loop is homotopically trivial by a
homotopy in $St^{M}(v)$ for $v$ any vertex of that loop. Combining these
homotopies gives a ladder homotopy.
\end{remark}

We briefly recall the outline of $\S $\ref{out}. We determine a compact set
$E(C_{0},C)$ such that for any compact set $F$, loops outside of $E$ and based
on a proper base ray $r_{0}$ can be pushed outside $F$ relative to $r_{0}$ and
by a homotopy avoiding $C_{0}$. A loop outside $E$ is written in the form
\[
\alpha= (\alpha_{1}, e_{1},\beta_{1}, \tilde e_{1},\alpha_{2},e_{2}\beta_{2},
\tilde e_{2}\ldots, \alpha_{n-1}, e_{n-1} ,\beta_{n-1}, \tilde e_{n-1},
\alpha_{n})
\]
where $\alpha_{i}$ is an edge path in $J\cdot C$, $e_{i}$ (respectively
$\tilde e_{i}$) is an edge with terminal (respectively initial) vertex in
$Y-J\cdot C$ and $\beta_{i}$ is an edge path in $Y-J\cdot C$ (see Figure 6).

We can push the $\alpha_{j}$ subpaths of $\alpha$ arbitrarily far out between
$(\tilde\gamma_{j-1}^{-1}, \tilde r_{j-1})$ and $(\gamma_{j}^{-1},r_{j}$)
using the semistability of $J$ in $Y$ with respect to $C$. Lemmas \ref{far}
and \ref{USS} consider subpaths of the form $(e,\beta, \tilde e)$ in $\alpha$.
The edges $e$ and $\tilde e$ are properly pushed off to infinity using ladder
homotopies given by Lemma \ref{step2}. The $\beta$ paths present difficulties
and two cases are considered. If $\beta$ lies in $gU_{i}$ and $gS_{i}\ast$
does not intersect $St^{M_{\ref{step2}}}(C)$ then Lemma \ref{far}, provides a
proper homotopy to compatibly push $(e,\beta, \tilde e)$ arbitrarily far out.
In Lemma \ref{USS} we consider paths $(e,\beta, \tilde e) $ not considered in
Lemma \ref{far}. For $g\in J$ and $i\in\{1,\ldots,l\} $ there are only
finitely many cosets $gS_{i}$ such that $(gS_{i}\ast)\cap St^{M_{\ref{step2}}%
}(C)\ne\emptyset$ and we are reduced to considering paths $(e,\beta, \tilde
e)$ with $\beta$ in $gU_{i}$ for these $gS_{i}$.

\begin{lemma}
\label{far} Suppose that $g\in J$, $i\in\{1,\ldots, l\}$ and $([w,v],\beta,
[\tilde v,\tilde w])$ is an edge path in $Y-D_{\ref{step2}}$. Suppose further that

1) $w,\tilde w\in J\cdot C$ and $v,\tilde v\in gU_{i}$,

2) $\beta$ is an edge path in $gU_{i}$,

3) $\gamma$ (respectively $\tilde\gamma$) is an edge path of length $\leq
N_{\ref{step1}}$ from $x =gx^{\prime}\in gS_{i}\ast$ (resp. $\tilde
x= g\tilde x^{\prime}\in gS_{i}$) to $w$ (resp. $\tilde w$), (such paths
exist by Lemma \ref{step1}) and

4) $x^{\prime}$ and $\tilde x^{\prime}$ belong to the \textbf{same} unbounded
component $Q$ of $\Lambda(S_{i},S_{i}^{0})-m_{(g,i)}^{-1} (St^{M_{\ref{step2}%
}}(C))$ (in particular, when $m_{(g,i)}^{-1} (St^{M_{\ref{step2}}%
}(C))=\emptyset$) then:

There are proper $\Lambda_{i}$-edge path rays $r^{\prime}$ at $x ^{\prime}$
and $\tilde r ^{\prime}$ at $\tilde x ^{\prime}$ such that, $r ^{\prime}$ and
$\tilde r ^{\prime}$ have image in $Q$ and if $r = m_{(g,i)}(r^{\prime})$
and $\tilde r = m_{(g,i)}(\tilde r ^{\prime})$ then for any compact set
$F\subset Y$, there is an integer $d\geq0$ and edge path $\psi$ in $Y-F$ from
$r (d)$ to $\tilde r (d)$ such that the loop:
\[
(r |_{[0,d]}^{-1}, \gamma,[ w ,v ], \beta, [\tilde v ,\tilde w ], \tilde
\gamma^{-1} , \tilde r |_{[0,d]},\psi^{-1})
\]
is homotopically trivial by a homotopy in $Y-C_{0}$. (So $([w,v], \beta,
[\tilde v,\tilde w])$ can be pushed between $(\gamma^{-1}, r)$ and
$(\tilde\gamma^{-1}, \tilde r)$ to a path in $Y-F$, by a homotopy in $Y-C_{0}$.)
\end{lemma}

\begin{proof}
Let $r ^{\prime}$ be any proper edge path in $Q$ with initial point $x
^{\prime}$. Let $\tau^{\prime}=(e ^{\prime},\ldots, e_{k}^{\prime})$ be an
edge path in $Q$ from $\tilde x^{\prime}$ to $x ^{\prime}$ with consecutive
vertices $(\tilde x ^{\prime}=t_{0}^{\prime},t ^{\prime},\ldots, t_{k}%
^{\prime}=x^{\prime})$. Let $\tilde r^{\prime}=(\tau^{\prime},r ^{\prime}) $.
Let $t_{j}=m_{(g,i)}(t_{j}^{\prime})$ for all $j\in\{0,1,\ldots, k\}$,
$r=m_{(g,i)}(r^{\prime})$, $\tilde r=m_{(g,i)}(\tilde r^{\prime})$ and
$\tau=m_{(g,i)}(\tau^{\prime})$ (an edge path from $\tilde x$ to $x$ with
image in $Y-St^{M_{\ref{step2}}}(C)$).

By Lemma \ref{step1} and the definition of $M_{ \ref{step2}}$, there is an
edge path $\delta$ in $gU_{i}$ from $\tilde v$ to $v $ such that the loop
$([\tilde v,\tilde w],\tilde\gamma^{-1}, \tau, \gamma,[w ,v ],\delta^{-1})$ is
homotopically trivial by a ladder homotopy $H_{1}$ (with rungs connecting the
two sides $\tau$ and $\delta$ and) with image in $St^{M_{\ref{step2}}}%
(\{t_{0},t ,\ldots, t_{k}\})\subset Y-C$.

By Lemma \ref{step2}, there is a proper edge path ray $s$ at $v$ and with
image in $gU_{i}$ such that $r$ is properly homotopic $rel\{x\}$ to
$(\gamma,[w,v],s)$ by a ladder homotopy $H_{2}$ in $Y-C$. Since $J$ is
co-semistable at $\infty$ in $Y$ with respect to $C_{0}$ and $C$ (and $s$ is
$J$-bounded), the loop $(\beta,\delta)$ can be pushed along $s$ by a homotopy
$H_{3}$ (with image in $Y-C_{0}$) to a loop $\phi$ in $Y-F$, where if $\phi$
is based at $s(k)$, then $s([k,\infty))$ avoids $F$.

\vspace{.5in} \vbox to 2in{\vspace {-2in} \hspace {-.5in}
\hspace{-.5 in}
\includegraphics[scale=1]{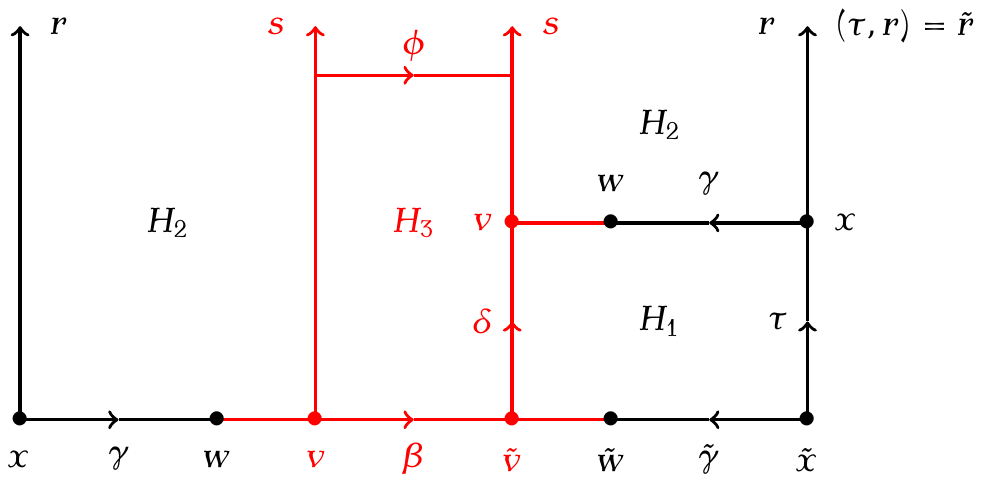}
\vss }

\vspace{-.3in}

\centerline{Figure 4}

\medskip

\medskip

Combine these homotopies
as in Figure 4 to obtain $\psi$.
\end{proof}

If $U$ is a $J$-unbounded component of $Y-J\cdot C$, and $s$ and $\tilde s$
are proper edge path rays in $Y$ and with image in $U$, then we say $s$ and
$\tilde s$ \textit{converge to the same end of $U$} (in $Y$) if for any
compact set $F$ in $Y$, there are edge paths in $U-F$ connecting $s$ and
$\tilde s$. Figure 6 can serve as a visual aid for Lemma \ref{last}.

\begin{lemma}
\label{last} There is a compact set $D_{\ref{last}}(C, U_{1},\ldots, U_{l})$
such that:

If $g\in J$, $i\in\{1,\ldots, l\}$, and $([w,v],\beta, [\tilde v,\tilde w])$
is an edge path in $Y-D_{\ref{last}}$ with $w,\tilde w\in J\cdot C$ and
$\beta$ a path in $gU_{i}$, then there are edge paths $\gamma$ and
$\tilde\gamma$ of length $\leq N_{\ref{step1}}$ from $x=gx^{\prime}\ast\in
gS_{i}\ast$ to $w$ and $\tilde x=g\tilde x^{\prime}\ast\in gS_{i}\ast$ to
$\tilde w$ respectively, and proper edge path rays $r^{\prime}$ at $x^{\prime
}$ and $\tilde r^{\prime}$ at $\tilde x^{\prime}$ with image in $\Lambda
(S_{i},S_{i}^{0})-m_{(g,i)}^{-1}(D_{\ref{step2}})$ such that for $r=
m_{(g,i)}(r^{\prime})$ and $\tilde r= m_{(g,i)}(\tilde r^{\prime})$, one
of the following two statements is true:

\begin{enumerate}
\item For any compact set $F$ in $Y$, there is an integer $d\in[0,\infty)$ and
edge path $\psi$ in $Y-F$ from $r(d)$ to $\tilde r(d)$ such that the loop
\[
(r |_{[0,d]}^{-1}, \gamma,[ w ,v ], \beta, [\tilde v ,\tilde w ], \tilde
\gamma^{-1}, \tilde r |_{[0,d]},\psi^{-1})
\]
is homotopically trivial by a homotopy in $Y-C_{0}$.

\item There are proper $J$-bounded edge path rays $s$ at $v$ and $\tilde s$ at
$\tilde v$ with image in $gU_{i}$ such that, the ray $s$ (respectively $\tilde
s$) is properly homotopic $rel\{v\}$ to $([v,w],\gamma^{-1}, r)$ (respectively
$rel\{\tilde v\}$ to $([\tilde v,\tilde w],\tilde\gamma^{-1}, \tilde r)$ by a
(ladder) homotopy in $Y-C$ (just as in Lemma \ref{step2}), and $s$ and $\tilde
s$ converge to the same end of $g U_{i}$.
\end{enumerate}
\end{lemma}

\begin{proof}
We define $D_{\ref{last}}$ to be the union of a finite collection of compact
sets. The first is $D=D_{\ref{step2}}(C)$ (which contains $St^{M_{\ref{step2}%
}}(C)$). If $\Lambda(S_{i},S_{i}^{0})-m_{(g,i)}^{-1}(St^{M_{\ref{step2}}}(C))$
has only one unbounded component (in particular when $m_{g,i}^{-1}%
(St^{M_{\ref{step2}}}(C))=\emptyset$) then conclusion 1) is satisfied (by
Lemma \ref{far}). There are only finitely many pairs $(g,i)$ with $g\in J$ and
$i\in\{1,\ldots,l\}$ such that $\Lambda(S_{i},S_{i}^{0})-m_{g,i}%
^{-1}(St^{M_{\ref{step2}}}(C))$ has more than one unbounded component. List
these pairs as $(g(1),\iota(1)),\ldots, (g(t),\iota(t))$. Now assume that
$gU_{i}=g(q)U_{\iota(q)}$ for some $q\in\{1,\ldots,t\}$. There are finitely
many unbounded components of $\Lambda(S_{i},S_{i}^{0})-m_{(g,i)}%
^{-1}(St^{M_{\ref{step2}}}(C))$. List them as $K_{1},\ldots, K_{a}$. Consider
pairs $(K_{j},K_{k})$ with $j\ne k$.

If for every compact set $F$ in $Y$, there are vertices $y_{j}^{\prime}\in
K_{j}$ and $y_{k}^{\prime}\in K_{k}$, edge paths $\tau_{j}$ and $\tau_{k}$ of
length $\leq N_{\ref{step1}}$ from $m_{(g,i)}(y_{j}^{\prime})$ to $gU_{i}$ and
$m_{(g,i)}(y_{k}^{\prime})$ to $gU_{i}$ respectively, and an edge path in
$gU_{i}-F$ connecting the terminal point of $\tau_{j}$ and the terminal point
of $\tau_{k}$, then we call the pair $(K_{j}, K_{k})$ \textit{inseparable} and
let $F_{(j,k)}=\emptyset$. Otherwise, we call the pair \textit{separable} and
let $F_{(j,k)}$ be the compact subset of $Y$ for which this condition fails.
Let $E_{(g,i)}=\cup_{j\ne k} F_{(j,k)}$. As $gU_{i}=g_{q}U_{\iota(q)}$, define
$E^{q}=E_{(g,i)}$.

We now define $D_{\ref{last}}=D_{\ref{step2}}(C)\cup E^{1}\cup\cdots\cup
E^{t}$. As noted above we need only consider the case where $\beta$ has image
in $g(q)U_{\iota(q)}$ for some $q\in\{1,\ldots, t\}$. Simplifying notation
again let $g=g(q)$ and $U_{i}=U_{\iota(q)}$. Lemma \ref{step1} implies there
are edge paths $\gamma$ and $\tilde\gamma$ of length $\leq N_{\ref{step1}}$
from $x= gx^{\prime}\ast\in gS_{i}\ast$ to $w$ and $\tilde x=
g\tilde x^{\prime}\ast\in gS_{i}\ast$ to $\tilde w$ respectively. Again let
$K_{1},\ldots, K_{a}$ be the unbounded components of $\Lambda(S_{i},S_{i}%
^{0})-m_{(g,i)}^{-1}(St^{M_{\ref{step2}}}(C))$. Assume that $x^{\prime}$
belongs to $K_{1}$. If $\tilde x^{\prime}$ also belongs to $K_{1}$, then
conclusion 1) of our lemma follows directly from Lemma \ref{far}.

So, we may assume $\tilde x^{\prime}$ belongs to $K_{2}\ne K_{1}$. Notice that
the existence of $\beta$ (in $Y-D_{\ref{last}}$) implies that the pair
$(K_{1},K_{2})$ is inseparable. This implies that there is a sequence of pairs
of vertices $(y_{1(j)}^{\prime}, y_{2(j)}^{\prime})$ for $j\in\{1,2,\ldots\}$
with $y_{1(j)}^{\prime}\in K_{1}$, $y_{2(j)}^{\prime}\in K_{2}$ and edge paths
$\tau_{1(j)}$ and $\tau_{2(j)}$ of length $\leq N_{\ref{step1}}$ from
$m_{(g,i)}(y_{1(j)}^{\prime})$ to $gU_{i}$ and $m_{(g,i)}(y_{2(j)}^{\prime})$
to $gU_{i}$ respectively, and an edge path $\beta_{j}$ in $gU_{i}$ from the
terminal point of $\tau_{1(j)}$ to the terminal point of $\tau_{2(j)}$ and
such that only finitely may $\beta_{j}$ intersect any compact set. Pick proper
edge path rays $r^{\prime}$ in $K_{1}$ at $x^{\prime}$ and $\tilde r^{\prime}$
in $K_{2}$ at $\tilde x^{\prime}$ so that for infinitely many pairs
$(y_{1(j)}^{\prime}, y_{2(j)}^{\prime})$, $r^{\prime}$ passes through
$y_{1(j)}^{\prime}$ and $\tilde r^{\prime}$ passes through $y_{2(j)}^{\prime}%
$. Let $r=m_{(g,i)}(r^{\prime})$ and $\tilde r=m_{(g,i)}(r^{\prime})$. Choose
$s$ and $\tilde s$ for $r$ and $\tilde r$ respectively as in Lemma \ref{step2}
where $\gamma$ and $\tilde\gamma$ for $r$ and $\tilde r$ are chosen to be
$\tau_{1(j)}$ and $\tau_{2(j)}$ when ever possible. Lemma \ref{step2} implies
the ray $s$ is properly homotopic $rel\{v\}$ to $([v,w],\gamma_{w}^{-1}, r)$
and $\tilde s$ is properly homotopic $rel\{\tilde v\}$ to $([\tilde v,\tilde
w],\tilde\gamma^{-1}, \tilde r)$ by ladder homotopies in $Y-C$. The paths
$\beta_{j}$ show that $s$ and $\tilde s$ converge to the same end of $g U_{i}%
$, so that conclusion 2) of our lemma is satisfied.
\end{proof}

\begin{lemma}
\label{USS} Suppose $U$ is a $J$-unbounded component of $Y-J\cdot C$, $F$ is
any compact subset of $Y$ and $s_{1}$ and $s_{2}$ are $J$-bounded proper edge
path rays in $U$ determining the same end of $U$, and with $s_{1}(0)=s_{2}%
(0)$, then there is an integer $n$ and a path $\beta$ from the vertex
$s_{1}(n)$ to the vertex $s_{2}(n)$ such that the image of $\beta$ is in $Y-F$
and $(s_{1}|_{[0,n]}, \beta, s_{2}|_{[0,n]}^{-1})$ is homotopically trivial in
$Y-C_{0}$.
\end{lemma}

\begin{proof}
Choose an integer $n$ such that $s_{1}([n,\infty))$ and $s_{2}([n,\infty))$
avoid $F$. Since $s_{1}$ and $s_{2}$ determine the same end of $U$, there is
an edge path $\alpha$ in $U-F$ from $s_{1}(n)$ to $s_{2}(n)$. Consider the
loop $(s_{1}|_{[0,n]}^{-1},s_{2}|_{[0,n]},\alpha^{-1})$ based at
$s_{1}|_{[n,\infty)}$. By co-semistability, there is a homotopy $H:[0,1]\times
\lbrack0,l]\rightarrow Y-C_{0}$ (see Figure 5) such that
\[
H(0,t)=H(1,t)=s_{1}(n+t)\hbox{ for }t\in\lbrack0,l],\ H(t,l)\in
Y-F\hbox{ for }t\in\lbrack0,1]\hbox{ and}
\]%
\[
H|_{[0,1]\times\{0\}}=(s_{1}|_{[0,n]}^{-1},s_{2}|_{[0,n]},\alpha^{-1})
\]

\vspace{.5in} \vbox to 2in{\vspace {-2in} \hspace {-1in}
\hspace{-.6 in}
\includegraphics[scale=1]{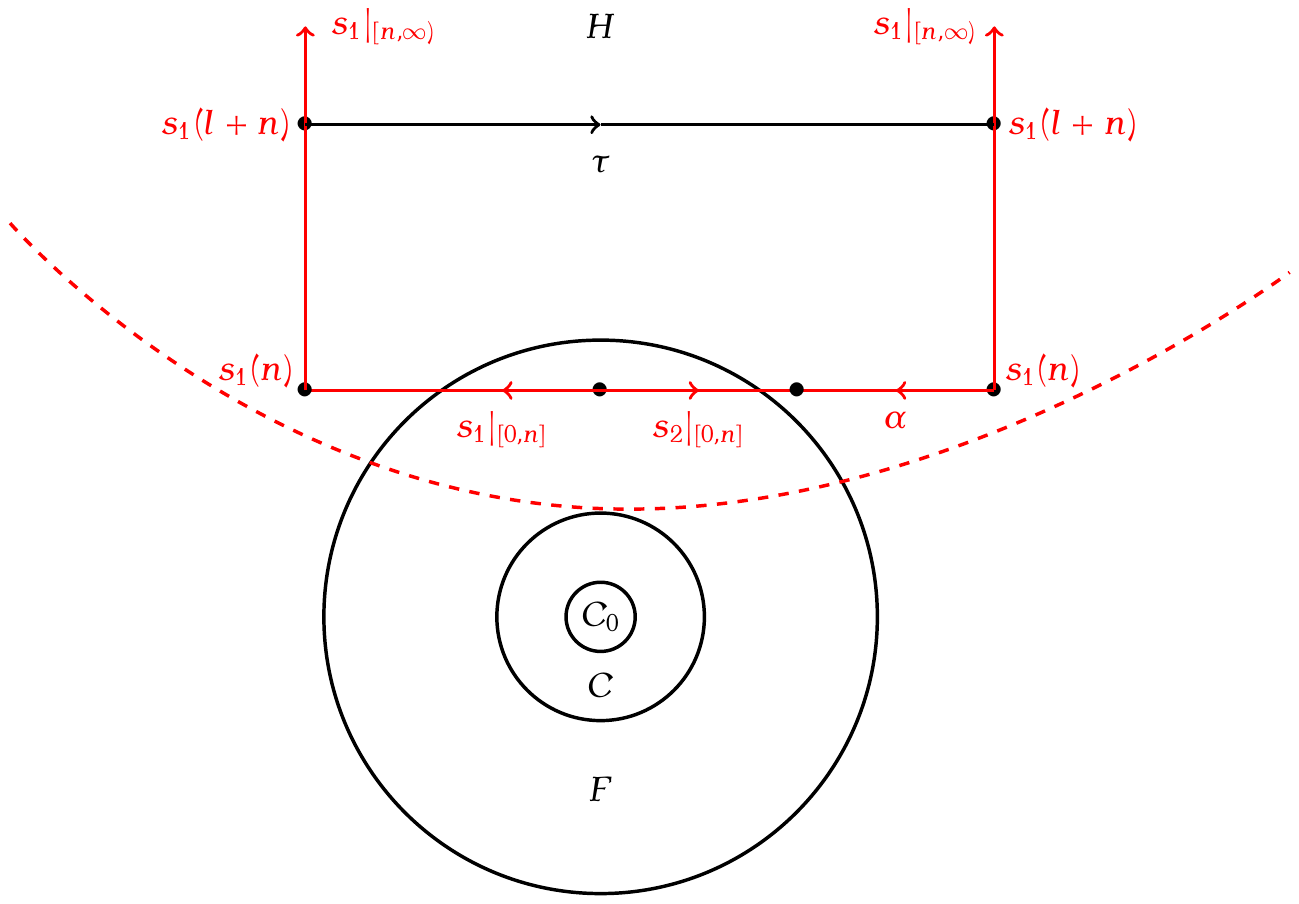}
\vss }

\vspace{1.5in}

\centerline{Figure 5}

\medskip

\medskip

Define $\tau(t)=H(t,l)$ for $t\in[0,1]$ (so that $\tau(0)=\tau(1)=s_{1}%
(l+n))$. Now define
\[
\beta=(s_{1}|_{[n,n+l]}, \tau, s_{1}|_{[n,n+l]}^{-1}, \alpha)
\]
to finish the proof.
\end{proof}

\begin{lemma}
\label{easy} Suppose $r_{1}^{\prime}$ and $r_{2}^{\prime}$ are proper edge
path rays in $\Lambda(J,J^{0})$ such that $m_{(g,i)}(r_{1}^{\prime})=
r_{1}$ and $m_{(g,i)}(r_{2}^{\prime})= r_{2}$ have image in $Y-C$. There
is a compact set $D_{\ref{easy}}(C)$ in $Y$ such that: if $\alpha$ is an edge
path in $(J\cdot C)\cap(Y-D_{\ref{easy}})$ from $r_{1}(0)$ to $r_{2}(0)$ and
$F$ is any compact set in $Y$, then there is an edge path $\psi$ in $Y-F$ from
$r_{1}$ to $r_{2}$ such that the loop determined by $\psi$, $\alpha$ and the
initial segments of $r_{1}$ and $r_{2}$ is homotopically trivial in $Y-C_{0}$.
\end{lemma}

\begin{proof}
There is an integer $N_{\ref{easy}}(C)$ such that for each vertex $v$ of $C$
there is an edge path in $Y$ from $v$ to $\ast$ of length $\leq N_{\ref{easy}%
}$. Then for each vertex $v$ of $J\cdot C$ there is an edge path of length
$\leq N_{\ref{easy}}$ from $v$ to $J\ast$. Choose an integer $P$ such that if
$v^{\prime}$ and $w^{\prime}$ are vertices of $\Lambda(J,J^{0})$ and
$z(v^{\prime})= v$ and $z(w^{\prime})= w$ are connected by an edge
path of length $\leq2N_{\ref{easy}}+1$ in $Y$ then $v^{\prime}$ and
$w^{\prime}$ are connected by an edge path of length $\leq P$ in
$\Lambda(J,J^{0})$. Recall that if $e$ is an edge of $\Lambda(J,J^{0})$ then
$z(e)$ is an edge path of length $\leq K$. By Lemma \ref{kill} there is an
integer $M_{\ref{easy}}$ such that any loop containing a vertex of $J\ast$ and
of length $\leq KP+ 2N_{\ref{easy}} + 1$ is homotopically trivial in
$St^{M_{\ref{easy}}}(v)$ for any vertex $v$ of this loop.

Let $D_{\ref{easy}}=St^{M_{\ref{easy}}}(C)$. Write $\alpha$ as the edge path
$(e_{1},\ldots, e_{p})$ with consecutive vertices $v_{0},v_{1},\ldots, v_{p}$.
Let $\beta_{0}$ and $\beta_{p}$ be trivial and for $i\in\{1,\ldots, p-1\}$ let
$\beta_{i}$ be an edge path of length $\leq N_{\ref{easy}}$ from $v_{i}$ to
some vertex $g_{i}\ast$ for $g_{i}\in J$. Let $g_{0}=r_{1}^{\prime}(0)$ and
$g_{p}=r_{2}^{\prime}(0)$ (so $g_{0}\ast=v_{0}$ and $g_{p}\ast=v_{p}$). For
$i\in\{0,\ldots, p-1\}$, there is an edge path $\tau_{i}^{\prime}$ in
$\Lambda(J,J^{0})$ from $g_{i-1}$ to $g_{i}$ of length $\leq P$. Let $\tau
_{i}= z(\tau_{i}^{\prime})$ (an edge path of length $\leq PK$. Then the
loop $(\beta_{i},\tau_{i+2},\beta_{i+1}^{-1},e_{i}^{-1})$ has length $\leq
KP+2N_{\ref{easy}}+1$ and so is homotopically trivial in $St^{M_{\ref{easy}}%
}(v)$ for any vertex $v$ of the loop. Let $\tau^{\prime}=(\tau_{1}^{\prime
},\ldots, \tau_{p}^{\prime})$, then $\alpha$ is homotopic $rel\{v_{0},v_{p}\}$
to $z(\tau^{\prime})=\tau$ by a (ladder) homotopy in $Y-C$. Since $J$ is
semistable at $\infty$ in $Y$ with respect to $J^{0}$, $C_{0}$ and $C$, there
is an edge path $\psi$ in $Y-F$ from $r_{1}$ to $(\tau,r_{2})$ such that the
loop determined by $\psi$, $\tau$ and the initial segments of $r_{1}$ and
$r_{2}$ is homotopically trivial in $Y-C_{0}$. Now combine this homotopy with
the homotopy of $\alpha$ and $\tau$.
\end{proof}

\begin{proof}
\textbf{(of Theorem \ref{MT})} Let $C_{0}$ be a finite subcomplex of $Y$ and
$J_{0}$ be a finite generating set for an infinite finitely generated group
$J$, where $J$ acts as cell preserving covering transformations on $Y$, $J$ is
semistable at $\infty$ in $Y$ with respect to $J_{0}$, $C_{0}$ and $C$ (a
finite subcomplex of $Y$) and $J$ is co-semistable at $\infty$ in $Y$ with
respect to $C_{0}$ and $C$. Also assume that $Y-J\cdot C$ is a union of
$J$-unbounded components. Let $U_{1},\ldots,U_{l}$ be $J$-unbounded components
of $Y-J\cdot C$ forming a component transversal for $Y-J\cdot C$ and let
$S_{i}$ be the $J$-stabilizer of $U_{i}$ for $i\in\{1,\ldots,l\}$. Let
$N_{\ref{step1}}$ be defined for $C$ and $U_{1},\ldots,U_{l}$ as in Lemma
\ref{step1}. Let $r_{0}^{\prime}$ be a proper edge path ray in $\Lambda
(J,J^{0})$ at $1$ and $r_{0}= zr_{0}^{\prime}$.

\vspace{.5in} \vbox to 2in{\vspace {-2in} \hspace {-1in}
\hspace{-.8 in}
\includegraphics[scale=1]{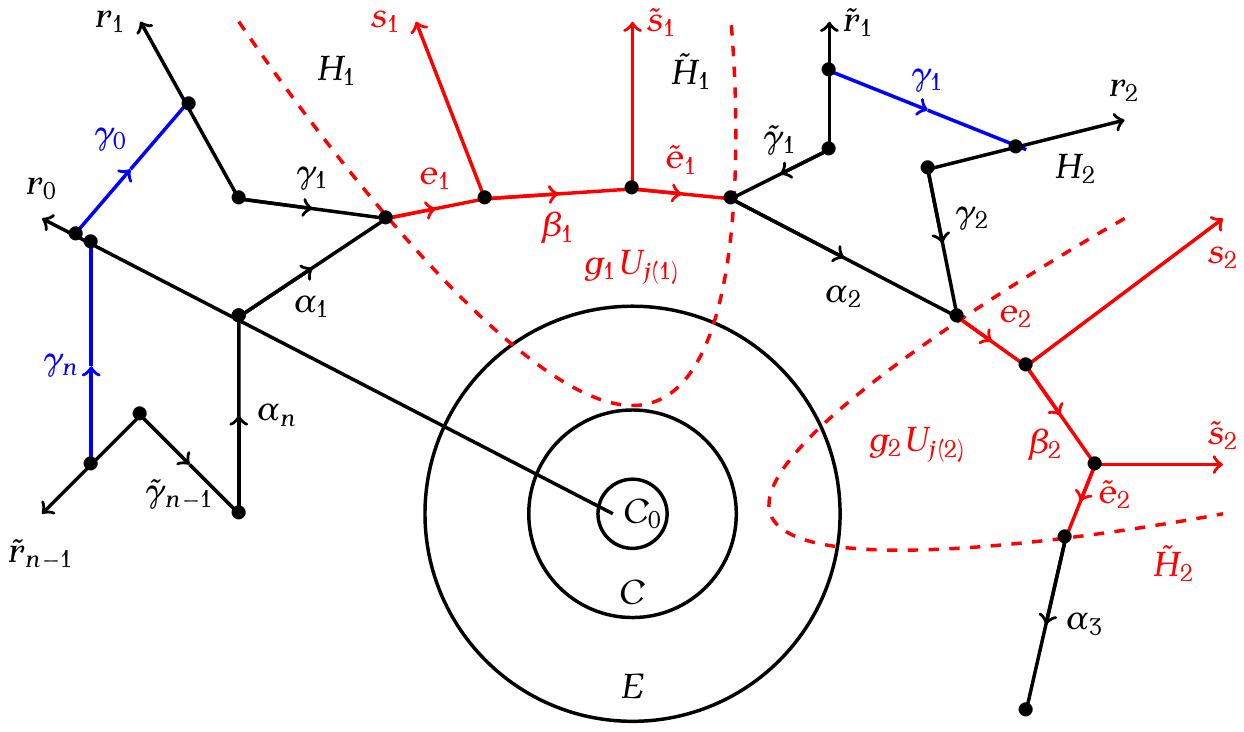}
\vss }

\vspace{.75in}

\centerline{Figure 6}

\medskip

\medskip

Let $E$ be compact containing $St^{N_{\ref{step1}}}(D_{\ref{easy}})\cup
D_{\ref{last}}(C,U_{1},\ldots, U_{l})$ and such that once $r_{0}$ leaves $E$
it never returns to $D_{\ref{last}}(C)$. Suppose $\alpha$ is an edge path loop
based on $r_{0}$ with image in $Y-E$ (see Figure 6). Let $F$ be any compact subset of $Y$.
Our goal is to find a proper homotopy $H:[0,1]\times[0,1]\to Y-C_{0}$ such
that $H(0,t)=H(1,t)$ is a subpath of $r_{0}$, $H(t,0)=\alpha$ and $H(t,1)$ has
image in $Y-F$ (so that $Y$ has semistable fundamental group at $\infty$ by
Theorem \ref{ssequiv} part 2).) Write $\alpha$ as:
\[
\alpha= (\alpha_{1}, e_{1},\beta_{1}, \tilde e_{1},\alpha_{2},e_{2}\beta
_{2},\tilde e_{2}\ldots, \alpha_{n-1}, e_{n-1} ,\beta_{n-1}, \tilde e_{n-1},
\alpha_{n})
\]
where $\alpha_{i}$ is an edge path in $J\cdot C$, $e_{i}$ (respectively
$\tilde e_{i}$) is an edge with terminal (respectively initial) vertex in
$Y-J\cdot C$ and $\beta_{i}$ is an edge path in the $J$-unbounded component
$g_{i}U_{f(i)}$ of $Y-J\cdot C$ where $f(i)\in\{1,\ldots, l\}$.

By Lemmas \ref{step1} and \ref{step2} and the definition of $D_{\ref{last}%
}(C)$, there is an edge path $\gamma_{i}$ of length $\leq N_{\ref{step1}}$,
from a vertex $x_{i}= gx_{i}^{\prime}\ast$ of $g_{i}S_{f(i)}\ast$ to the
initial vertex of $e_{i}$, and there are proper edge path rays $r_{i}^{\prime
}$ at $x_{i}^{\prime}$ in $\Lambda(S_{f(i)}, S_{f(i)}^{\prime})$ and $s_{i}$
at the end point of $e_{i}$ such that $s_{i}$ has image in $g_{i}U_{f(i)}$ and
$r_{i}$ is properly homotopic to $(\gamma_{i}, e_{i}, s_{i})$ (where
$r_{i}= m_{(g,f(i))}(r_{i}^{\prime})$), by a proper (ladder) homotopy
$H_{i}$ with image in $Y-C$. Similarly there is an edge path $\tilde\gamma
_{i}$ of length $\leq N_{\ref{step1}}$ from $\tilde x_{i}$, a vertex of
$g_{i}S_{j(i)}\ast$, to the terminal vertex of $\tilde e_{i}$, and there are
$J$-bounded proper edge path rays $\tilde r_{i}$ at $\tilde\gamma_{j}(0)$ and
$\tilde s_{i}$ at the initial point of $\tilde e_{i}$, such that $\tilde
r_{i}= m_{(g_{i},f(i))}(\tilde r_{i}^{\prime})$ for some proper ray
$\tilde r_{i}^{\prime}$ in $\Lambda(S_{f(i)}, S_{f(i)}^{\prime})$, $\tilde
s_{i}$ has image in $g_{i}U_{f(i)} $ and $\tilde s_{i}$ is properly homotopic
to $(\tilde e_{i}, {\tilde\gamma_{i}}^{-1}, \tilde r_{i})$ by a proper
(ladder) homotopy $\tilde H_{i}$ with image in $Y-C$. In particular, the
$r_{i}$, and $\tilde r_{i}$-rays have image in $Y-C$.

By Lemma \ref{last}, either $r_{i}$ is properly homotopic $rel \{r_{i}(0)\}$
to the ray $(\gamma_{i},e_{i},\beta_{i}, \tilde e_{i}, \tilde\gamma_{i}^{-1},
\tilde r_{i}) $ by a homotopy in $Y-C_{0}$ or the rays $s_{i}$ and $\tilde
s_{i}$ converge to the same end of $g_{i} U_{f(i)}$. In the former case: The
path $(\gamma_{i},e_{i},\beta_{i}, \tilde e_{i}, \tilde\gamma_{i}^{-1})$ can
be moved by a homotopy along $r_{i}$ and $\tilde r_{i}$ to a path outside $F$
where the homotopy has image in $Y-C_{0}$.

In the later case, Lemma \ref{USS} implies there is a there is an integer
$n_{i}$ and edge path $\tilde\beta_{i}$ from $s_{i}(n_{i})$ to $\tilde
s_{i}(n_{i})$ and with image in $Y-F$ such that $\beta_{i}$ can be moved by a
homotopy along $s_{i} $ and $\tilde s_{i}$ to $\tilde\beta_{i}$, such that
this homotopy has image in $Y-C_{0}$. In any case, the (ladder) homotopy
$H_{i}$ (of $r_{i}$ to $(\gamma_{i},e_{i},s_{i})$) tells us that $(\gamma
_{i},e_{i})$ can be moved (by a homotopy in $Y-C_{0}$) along $r_{i}$ and
$s_{i}$ to a path in $Y-F$ and similarly for $(\tilde\gamma_{i},\tilde e_{i})$
using $\tilde H_{i}$. Combining these three homotopies, we have in the latter
case (as in the former):

$\ast$) The path $(\gamma_{i},e_{i},\beta_{i}, \tilde e_{i}, \tilde\gamma
_{i}^{-1})$ can be moved by a homotopy along $r_{i}$ and $\tilde r_{i}$ to a
path outside $F$ by a homotopy with image in $Y-C_{0}$.

For consistent notation, let $\tilde r_{0}=r_{n}$ be the tail of $r_{0}$
beginning at $\alpha_{1}(0)$, and let $\tilde\gamma_{0}$ and $\gamma_{n}$ be
the trivial paths at the initial point of $\alpha_{1}$. It remains to show
that for $0\leq i\leq n$, there is a path $\delta_{i}$ in $Y-F$ from $\tilde
r_{i}$ to $r_{i+1}$ such that the loop determined by $\delta_{i}$, the path
$(\tilde\gamma_{i},\alpha_{i+1}, \gamma_{i+1}^{-1})$, and the initial segments
of $\tilde r_{i} $ and $r_{i+1}$ is homotopically trivial in $Y-C_{0}$. These
homotopies are given by Lemma \ref{easy} since the paths $\gamma_{i}$ and
$\tilde\gamma_{i}$ all have length $\leq N_{\ref{step1}}$ and so by the
definition of $E$ they have image in $Y-D_{\ref{easy}}$ (as do the $\alpha
_{i}$), and since the rays $r_{i}$ and $\tilde r_{i}$ have image in $Y-C$.
\end{proof}

\section{Generalizations to absolute neighborhood retracts}\label{ANR}

There is no need for a space $X$ to be a CW complex in order to define what it
means for a finitely generated group $J$ to be semistable at $\infty$ in $X$
with respect to a compact subset $C_{0}$ of $X$, or for $J$ to be
co-semistable at $\infty$ in $X$ with respect to $C_{0}$.

\begin{corollary}
\label{MC} Suppose $X$ is a 1-ended simply connected locally compact absolute
neighborhood retract (ANR) and $G$ is a group (not necessarily finitely
generated) acting as covering transformations on $X$. Assume that for each
compact subset $C_{0}$ of $X$ there is a finitely generated subgroup $J$ of
$G$ so that (a) $J$ is semistable at $\infty$ in $X$ with respect to $C_{0}$,
and (b) $J$ is co-semistable at $\infty$ in $X$ with respect to $C_{0}$. Then
$X$ has semistable fundamental group at $\infty$.
\end{corollary}

\begin{proof}
By a theorem of J. West \cite{West77} the locally compact ANR $G\backslash X$
is proper homotopy equivalent to a locally finite polyhedron $Y_{1}$. A
simplicial structure on $Y_{1}$ lifts to a simplicial structure structure on
$Y $, its universal cover, and $G$ acts as cell preserving covering
transformations on $Y$. A proper homotopy equivalence from $G\backslash X$ to
$Y_{1}$ lifts to a $G$-equivariant proper homotopy equivalence $h: X\to Y$.
Let $f:Y\to X$ be a ($G$-equivariant) proper homotopy inverse of $h$. Since
the semistability of the fundamental group at $\infty$ of a space is invariant
under proper homotopy equivalence it suffices to show that $Y$ satisfies the
hypothesis of Theorem \ref{MT}.

First we show that if $C_{0}$ is compact in $Y$ then there is a finitely
generated subgroup $J$ of $G$ such that $J$ is semistable at $\infty$ in $Y$
with respect to $C_{0}$. There is a finitely generated subgroup $J$ of $G$,
with finite generating set $J^{0}$ and compact set $C\subset X$ such that $J$
is semistable at $\infty$ with respect to $J^{0}$, $h^{-1}(C_{0})$, $C$ and
$z_{1}$, where $z_{1}:\Lambda(J,J_{0})\to X$ is $J$-equivariant. Note that
$z=hz_{1}$ is $J$-equivariant. Let $r^{\prime}$ and $s^{\prime}$ be proper
edge path rays in $\Lambda$ such that $r^{\prime}(0)=s^{\prime}(0)$ and both
$r=z_{1}(r^{\prime})$ and $s=z_{1}(s^{\prime})$ have image in $X-C$. Then
given any compact set $D$ in $X$ there is path $\delta_{D}$ in $X-D$ from $r$
to $s $ such that the loop determined by $\delta_{D}$ and the initial segments
of $r$ and $s$ is homotopically trivial in $X-h^{-1}(C_{0})$.

Now, let $D$ be compact in $Y$. Suppose that $r^{\prime}$ and $s^{\prime}$ are
proper edge path rays in $\Lambda$ such that $r^{\prime}(0)=s^{\prime}(0) $
and both $r=hz_{1}(r^{\prime})$ and $s=hz_{1}(s^{\prime})$ have image in
$X-h(C)$ (in particular, $z_{1}(r^{\prime})$ and $z_{1}(s^{\prime})$ have
image in $X-C$). Let $\delta$ be a path from $z_{1}(r^{\prime})$ to
$z_{1}(s^{\prime})$ in $X-h^{-1}(D)$ (so that $h(\delta)$ is a path from $r$
to $s$ in $Y-D$) such that the loop determined by $\delta$ and the initial
segments of $z_{1}(r^{\prime})$ and $z_{1}(s^{\prime})$ is homotopically
trivial by a homotopy $H_{0}$ with image in $X-h^{-1}(C_{0})$. Then the loop
determined by $h(\delta)$ and the initial segments of $r$ and $s$ is
homotopically trivial in $Y-C_{0}$ by the homotopy $hH_{0}$.

Finally we show that if $C_{0}$ is compact in $Y$ there is a finitely
generated subgroup $J$ of $G$ such that $J$ is co-semistable at $\infty$ in
$Y$ with respect to $C_{0}$. Consider the compact set $h^{-1}(C_{0})\subset
X$. Choose $C$ compact in $X$ such that $J$ is co-semistable at $\infty$ in
$X$ with respect to $h^{-1}(C_{0})$ and $C$.

Let $H:Y\times[0,1]\to Y$ be a proper homotopy such that $H(y,0)=y$ and
$H(y,1)=hf(y)$ for all $y\in Y$. Let $D_{1}$ be compact in $Y$ so that if $s$
is a proper ray in $Y-D_{1}$ then the proper homotopy of $s$ to $hf(s)$
(induced by $H$) has image in $Y-C_{0}$. Let $D_{2}=D_{1}\cup f^{-1}(C)$. It
suffices to show that if $r$ is a $J$-bounded proper ray in $Y-J\cdot D_{2}$
and $\alpha$ is a loop in $Y-J\cdot D_{2}$ with initial point $r(0)$, then for
any compact set $F$ in $Y$, $\alpha$ can be pushed along $r$ to a loop in
$Y-F$, by a homotopy in $Y-C_{0}$. Define $\tau(t)= H(r(0),t)$) for
$t\in[0,1]$.

Let $H_{1}:[0,\infty)\times[0,1]\to Y-C_{0}$ be the proper homotopy (induced
by $H$) of the proper ray $(\alpha, r)$ to $(hf(\alpha), hf(r))$ so that
$H_{1}(t,0)=(\alpha,r)(t)$, $H_{1}(t,1)=(hf(\alpha), hf(r))(t)$ for
$t\in[0,\infty)$ and $H_{1}(0,t)=\tau(t)$ (see Figure 7). Let $H_{2}%
:[0,\infty)\times[0,1]\to Y-C_{0}$ be the proper homotopy (induced by $H$) of
$r$ to $hf(r)$ so that $H_{2}(t,0)=r(t)$, $H_{2}(t,1)=hf(r)(t)$ for
$t\in[0,\infty)$ and $H_{2}(0,t)=\tau(t)$ for $t\in[0,1]$.

\vspace{.5in} \vbox to 2in{\vspace {-2in} \hspace {-.5in}
\hspace{-.6 in}
\includegraphics[scale=1]{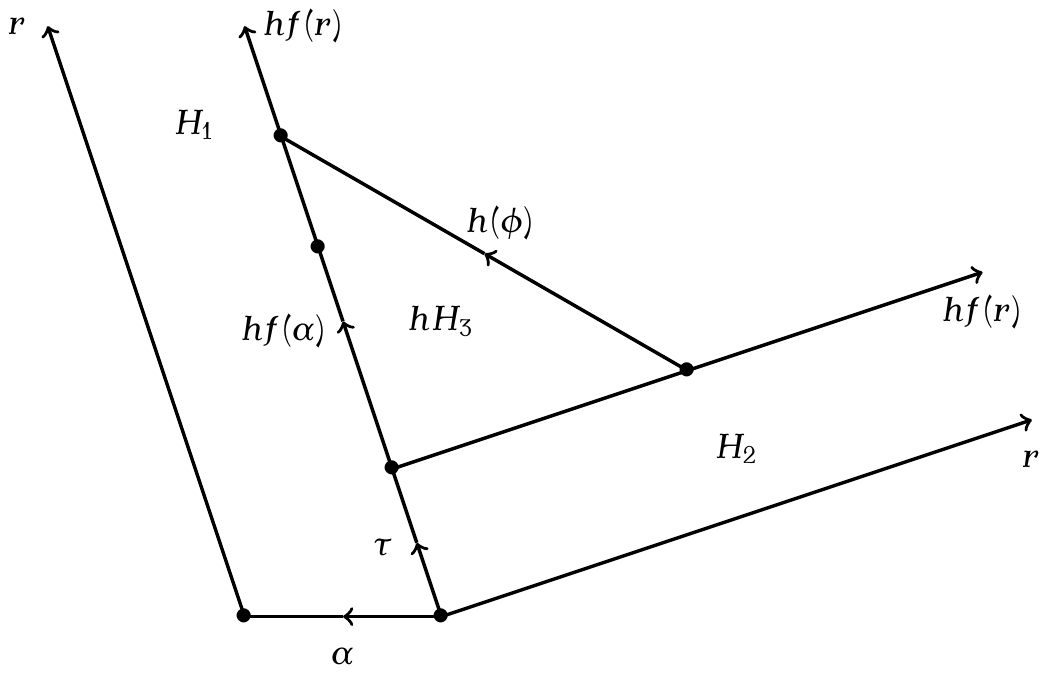}
\vss }

\vspace{.5in}

\centerline{Figure 7}

\medskip

\medskip

Recall that $f$ is $J$-equivariant. Since $r$ and $\alpha$ have image in
$Y-J\cdot D_{2}$ (and $f^{-1}(C)\subset D_{2}$), $f(r)$ and $f(\alpha)$ have
image in $X-J\cdot C$. Also $f(r)$ is $J$-bounded in $X$. There is a homotopy
$H_{3}$ with image in $X-h^{-1}(C_{0})$ that moves $f(\alpha)$ along $f(r)$ to
a loop $\phi$ in $X-h^{-1} (F)$, where if $fr(q)$ is the initial point of
$\phi$ then $fr([q,\infty))\subset X-h^{-1}(F)$. The homotopy $hH_{3}$ has
image in $Y-C_{0}$ and moves $hf(\alpha)$ along $hf(r)$ to the loop $h(\phi)$
in $Y-F$. Combine the homotopies $H_{1}$, $H_{2}$ and $H_{3}$ as in Figure 7
to see that $\alpha$ can be moved along $r$ into $Y-F$ by a homotopy in
$Y-C_{0}$.
\end{proof}

\bibliographystyle{amsalpha}
\bibliography{paper}
{}

\end{document}